\documentclass[a4paper,11pt]{article}

\usepackage[top=3.5cm, bottom=3.5cm, left=2.5cm, right=2.5cm]{geometry} 

\usepackage{fancyhdr}
\usepackage{titlesec}
\titleformat{\chapter}[hang]
  {\vspace{-1cm}\normalfont\huge\bfseries\Huge}
  {\chaptertitlename\ \thechapter}{10pt}{\,\,\,\,\sc\Huge}
\usepackage{titletoc}

\usepackage[sc,margin=1cm,font=small]{caption}

\usepackage{arydshln}
\usepackage{graphicx}
\usepackage{subfigure}
\usepackage{here}
\usepackage{float}
\usepackage{amsthm}
\usepackage{amssymb}
\usepackage{epstopdf}
\usepackage[latin1]{inputenc} 
\usepackage[english]{babel} 
\usepackage{amsmath}
\usepackage{amssymb}
\usepackage{mathrsfs}
\usepackage{enumerate}
\usepackage{color}
\usepackage[colorlinks=false,linkcolor=blue]{hyperref}
\usepackage[nottoc, notlof, notlot]{tocbibind}
\usepackage{titling}
\usepackage{mathtools}
\usepackage[]{pdfpages}
\usepackage[numbers,sort&compress]{natbib}
\newcommand{\subtitle}[1]{%
  \posttitle{%
    \par\end{center}
    \begin{center}\large#1\end{center}
    \vskip0.5em}%
}
\usepackage{pgf,tikz}
\usepackage{mathrsfs}
\usetikzlibrary{arrows}
\definecolor{uuuuuu}{rgb}{0.26666666666666666,0.26666666666666666,0.26666666666666666}
\usepackage{young}
\usepackage{authblk}

\DeclareFontFamily{U}{matha}{\hyphenchar\font45}
\DeclareFontShape{U}{matha}{m}{n}{
      <5> <6> <7> <8> <9> <10> gen * matha
      <10.95> matha10 <12> <14.4> <17.28> <20.74> <24.88> matha12
      }{}
\DeclareSymbolFont{matha}{U}{matha}{m}{n}
\DeclareFontSubstitution{U}{matha}{m}{n}

\DeclareFontFamily{U}{mathx}{\hyphenchar\font45}
\DeclareFontShape{U}{mathx}{m}{n}{
      <5> <6> <7> <8> <9> <10>
      <10.95> <12> <14.4> <17.28> <20.74> <24.88>
      mathx10
      }{}
\DeclareSymbolFont{mathx}{U}{mathx}{m}{n}
\DeclareFontSubstitution{U}{mathx}{m}{n}

\DeclareMathDelimiter{\vvvert}{0}{matha}{"7E}{mathx}{"17}

\newcommand{\R}{\mathbb{R}}

\newcommand{\N}{\mathbb{N}}

\newcommand{\HH}{\mathbb{H}}

\newcommand{\Sph}{\mathbb{S}}

\newcommand{\floor}[1]{\left\lfloor#1\right\rfloor}

\newcommand{\OO}{\mathcal{O}}

\renewcommand{\SS}{\mathscr{S}}

\newcommand{\1}{\mathbf{1}}
\newcommand{\MM}{\mathscr{M}}
\DeclareMathOperator{\Var}{Var}

\DeclareMathOperator{\Vect}{Span}

\DeclareMathOperator{\diag}{diag}

\DeclareMathOperator{\Card}{Card}

\DeclareMathOperator{\Orb}{Orb}

\DeclareMathOperator{\Tr}{Tr}

\DeclareMathOperator{\Span}{Span}

\DeclareMathOperator{\Hess}{Hess}
\addto\captionsenglish{%
}

\newtheoremstyle{exemple}
	{\topsep}
	{\topsep}
	{\normalfont}
	{}
	{\sc}
	{.}
	{5pt plus 1pt minus 1pt}
	{}

\newtheorem{theorem}{Theorem}[]

\newtheorem{proposition}{Proposition}[section]

\newtheorem{lemma}{Lemma}[section]

\newtheorem{corollary}{Corollary}[section]

\theoremstyle{definition}
\newtheorem{definition}{Definition}[section]

\theoremstyle{example}

\newtheorem{rem}{Remark}[section]

\title{\bf Phase transitions and macroscopic limits in a BGK model of body-attitude coordination}
\date{}
\author[1]{P. Degond}
\author[2]{A. Diez}
\author[3]{A. Frouvelle}
\author[4]{S. Merino-Aceituno}
\affil[1]{\footnotesize
Department of Mathematics, Imperial College London, South
Kensington Campus,
London, SW7 2AZ, UK,

\url{pdegond@imperial.ac.uk}}
\affil[2]{
Department of Mathematics, Imperial College London, South
Kensington Campus,
London, SW7 2AZ, UK,

\url{antoine.diez18@imperial.ac.uk}
}
\affil[3]{
CEREMADE, CNRS, Université Paris-Dauphine, Université PSL, 75016 Paris,
France,

\url{frouvelle@ceremade.dauphine.fr}
}
\affil[4]{
University of Vienna, Faculty of Mathematics, Oskar-Morgenstern-Platz 1, 1090 Wien, Austria

\url{sara.merino@univie.ac.at}

University of Sussex, Department of Mathematics, Falmer BN1 9RH, UK,

\url{s.merino-aceituno@sussex.ac.uk}

}
\begin{document}
\renewcommand{\proofname}{\sc{Proof}}
\maketitle
\abstract{In this article we investigate the phase transition phenomena that occur in a model of self-organisation through body-attitude coordination. Here, the body-attitude of an agent is modelled by a rotation matrix in $\R^3$ as in \cite{degondfrouvellemerino17}. The starting point of this study is a BGK equation modelling the evolution of the distribution function of the system at a kinetic level. The main novelty of this work is to show that in the spatially homogeneous case, self-organisation may appear or not depending on the local density of agents involved. We first exhibit a connection between body-orientation models and models of nematic alignment of polymers in higher dimensional space from which we deduce the complete description of the possible equilibria  Then, thanks to a gradient-flow structure specific to this BGK model, we are able to prove the stability and the convergence towards the equilibria in the different regimes. We then derive the macroscopic models associated to the stable equilibria in the spirit of \cite{degondfrouvellemerino17} and \cite{degondfrouvelleliu15}.}\\



\tableofcontents

\section{Introduction}

The model studied in the present work is a new elaboration of the work initiated in \cite{degondfrouvellemerino17} to model collective behaviour of agents described by their position and body-attitude. New results about emergence of phenomena of body-attitude coordination are presented in the context of a Bhatnagar-Gross-Krook (BGK) model. Such models can be applied to many biological systems such as flocking birds \cite{hildenbrandt2010self}, fish school \cite{hemelrijk2010emergence,hemelrijk2012schools} or sperm motion \cite{degond2015multi}. These systems are constituted by a large number of self-propelled agents which move at a constant speed and try to imitate their neighbours by moving in the same direction and trying to coordinate their body attitude. The agents are modelled by a moving frame in dimension~3, i.e. three orthogonal axes, one of which gives the direction of the motion and the two others the body orientation. In this work, as in~\cite{degondfrouvellemerino17}, the body attitude is modelled by a rotation matrix in dimension 3, i.e. an element of the special orthogonal group $SO_3(\R)$. In \cite{degondfrouvellemerinotrescases18} agents are modelled by quaternions. \\

Collective behaviour in many-agent systems has been a thoroughly studied subject in the mathematical literature, from both theoretical and applied points of view. Among the models which have received the most attention, one can cite the Cucker-Smale model \cite{cucker2007emergent,ha2009simple,motsch2011new}, attractive-repulsive models \cite{carrillo2017review} or the Vicsek model for self-propelled particles \cite{vicsek1995novel}. The present work belongs to the class of Vicsek-inspired models. Such models have two main distinctive features, first the assumption that the particles are self-propelled and secondly a geometrical constraint: in the original work of Vicsek, the velocities of the particles have constant norm and the dynamics therefore takes place on the sphere $\Sph^{n-1}$ in dimension~$n$ ($n=2$ in \cite{vicsek1995novel}). Here the dynamics takes place on the Riemannian manifold $SO_3(\R)$. \\ 

The tools used to study models of collective behaviour are generally borrowed from the mathematical kinetic theory of gases which gives a mathematical framework to study many-particle systems. At a microscopic scale, the motion of each particle is detailed (Individual Based Model, IBM) through Ordinary Differential Equations (ODE) coming from Newton's laws or through stochastic processes. When the number of particles is large, the whole system is described at a mesoscopic scale by a kinetic partial differential equation such as the Boltzmann, Fokker-Planck or BGK equation. Finally, large-scale dynamics is described by macroscopic equations (Euler, Navier-Stokes\ldots). A review of the main results of kinetic theory of gases can be found in \cite{degond2004macroscopic}. In particular the BGK equation (for Bhatnagar-Gross-Krook) was introduced in \cite{bhatnagar1954model} as a substitute for the Boltzmann equation in the context of gas dynamics. The BGK operator is a relaxation operator towards a Maxwellian having the same moments as the distribution function of the system. Its mathematical properties and relevance in the mathematical kinetic theory of gases have been studied in particular in \cite{perthame1989global} and \cite{saint2003bgk}. The BGK operator has been used in a model of collective dynamics of self-propelled particles in \cite{dimarcomotsch16}. However, together with \cite{degondfrouvellemerinotrescasesAlignment18}, it is the first time that it is rigorously studied in a body-attitude coordination model.\\ 

The main mathematical challenge in classical kinetic theory is the rigorous derivation of the kinetic equations from the IBM and of the macroscopic models from the kinetic equations. These questions are at the core of Hilbert's sixth problem and have received much attention in the last decades. Many different techniques have been developed to derive kinetic equations from hard-sphere gases (Boltzmann-Grad limit \cite{lanford1975time,gallagher2013newton}), from systems of interacting particles (mean-field limit and propagation of chaos \cite{jabin2014review,hauray2007n,sznitman89}) or from stochastic processes (and in particular jump processes \cite{mischler2013kac,kac1956foundations}). Some of these techniques have been adapted to problems arising in the study of collective behaviour \cite{chuang2007state,bolley2011stochastic,bolley2012mean}. The passage from kinetic equations to macroscopic models generally depends on physical constraints and in particular on conservation laws (hydrodynamic limits, Hilbert and Chapman-Enskog methods, see \cite{cercignani2013mathematical,degond2004macroscopic} for a review) and is still an active research field \cite{guo2010global,esposito2018stationary,caflisch1980fluid,golse2004navier}.
In the context of self-propelled particles, due to the lack of conservation laws which normally hold in the classical kinetic theory of gases, specific tools are needed. In \cite{degondmotsch08}, a methodological breakthrough has been achieved by introducing the so called Generalised Collisional Invariants (GCI) to rigorously link kinetic and macroscopic equations in the context of collective behaviour of self-propelled particles. This technique is now rigorously justified \cite{jiang2016hydrodynamic} and has already been successfully applied to a wide range of problems \cite{jiang2017coupled,zhang2017local}. It will be the key here to derive the macroscopic model in Section \ref{macroscopic}. This will lead to a system of partial differential equations on the mean density and body attitude, referred as the Self-Organised Hydrodynamics for Body-attitude coordination (SOHB) in \cite{degondfrouvellemerino17}.    \\ 

The aim of this work is to show the emergence of collective behaviour and self-organisation which give rise to macroscopic scale patterns such as clusters, travelling bands etc. These patterns emerge from the collective interactions and are not directly encoded in the behaviour of the individual particles as described by the IBM. The continuum version of the Vicsek model \cite{degondmotsch08} named the Self-Organised Hydrodynamics (SOH) model is an exemple of a model able to describe such emergence of self-organised dynamics. The Vicsek model describes a system where agents try to imitate their neighbours by adapting their direction of motion to the average direction of their neighbours. It has been shown that, in a certain scaling and when the equilibrium of the system is reached, the directions of motion of the agents are not uniformly distributed but follow a von Mises distribution. For $\kappa\in\R_+$ and $\Omega\in\Sph^{n-1}$ the von Mises distribution of parameters $\kappa$ and $\Omega$ is the Probability Density Function (PDF) on $\Sph^{n-1}$ defined by:
\[ M_{\kappa\Omega}(\omega):=\frac{e^{\kappa\Omega\cdot\omega}}{\int_{\Sph^{n-1}}e^{\kappa\Omega\cdot\omega'}\,d\omega'},\]
where the dot product is the usual dot product in $\R^n$. This model \cite{degondmotsch08} has been the starting point of many other models of self-organised dynamics, including \cite{degondfrouvellemerino17} for the body-attitude coordination. In this context, we define the von Mises distribution of parameter $J\in\MM_3(\R)$ (a $3\times3$ real matrix) as the following PDF on $SO_3(\R)$:
\[M_J(A):=\frac{e^{J\cdot A}}{\int_{SO_3(\R)} e^{J\cdot A'}\,dA'},\]
where the dot product and the measure on $SO_3(\R)$ come from the Riemannian structure of $SO_3(\R)$ detailed in Section \ref{SO3technique}.\\

In the present work, we focus on phase transition phenomena between non-organised and organised dynamics (collective motion). We will prove that the spatial density of agents is the key parameter which encodes the main features of phase transitions: in low density regions, no self-organised dynamics appears but when the density crosses a critical value, self-organised dynamics, given by a von Mises distribution for the body-attitude, becomes a stable equilibria of the system. This phase transition in the dynamics is purely an emergent phenomena, in the sense that at the macroscopic scale, different equations are required to describe the dynamics for different values of the density of agents, whereas for the IBM and at a mesoscopic level, the dynamics is described by one unique (system of) equation(s).\\

The starting point of this study is the BGK equation 
\begin{equation*}
\partial_t f+(Ae_1\cdot\nabla_x)f=\rho_f M_{J_f}-f,\end{equation*}
where $f(t,x,A)$ is a probability measure which gives the distribution of agents at position $x\in\R^3$ with body-orientation $A\in SO_3(\R)$ at time $t\in\R_+$ and where: 
\[\rho_f(t,x)=\int_{SO_3(\R)} f(t,x,A)\,dA\,\,\,\,\,\,\text{and}\,\,\,\,\,J_f=\int_{SO_3(\R)} f(t,x,A)A\,dA\]
are the respective local density and flux. The measure on $SO_3(\R)$ is the normalised Haar measure, the main properties of which are summarised in Section \ref{SO3technique}. \\

The left-hand side of the equation models the transport phenomenon: an agent with body orientation $A\in SO_3(\R)$ moves in the direction $Ae_1$ where $e_1$ is the first vector of the canonical basis of $\R^3$. The right-hand side of the equation is the BGK operator which models the interactions between the agents: here we assume that $f$ relaxes towards a ``moving equilibrium'' which takes the form of a von Mises distribution. In particular, the von Mises distribution appears as the analog of the Maxwellian distribution of the classical gas dynamics. The flux $J_f$ plays the same role as the momentum density for gas dynamics or the average flux for the Vicsek model. The term $\rho_f M_{J_f}$ can therefore be seen as the analog of the ``Maxwellian distribution with same moments as $f$'' in the context of the BGK equation for gas dynamics.\\

The main results of this work are (informally) summarised in the two following theorems. 

\begin{theorem}\label{theorem1} Let us consider the spatially homogeneous BGK equation:
\[\partial_t f=\rho M_{J_f}-f,\]
where $\rho\in\R_+$ is a given density of agents. 
\begin{enumerate}
\item The equilibria $f^\mathrm{eq}$ of the spatially homogeneous BGK model are either the uniform equilibrium $f^\mathrm{eq}=\rho$ or of the form $f^\mathrm{eq}=\rho M_{\alpha\Lambda}$ or $f^\mathrm{eq}=\rho M_{\alpha\,p\otimes q}$ where $\Lambda\in SO_3(\R)$ and $p,q\in \Sph^2$ and where $\alpha\in\R$ and $\rho$ are linked by a compatibility equation to be defined later (see Section \ref{paragraphequilibria} and equations \eqref{compatibilityc1} and \eqref{compatibilityc2}). 
\item Depending on the density of agents $\rho\in\R_+$, the only stable equilibria are either the uniform equilibrium $f^\mathrm{eq}=\rho$ or the equilibria of the form $f^\mathrm{eq}=\rho M_{\alpha\Lambda}$ where $\Lambda~\in~SO_3(\R)$ and where $\alpha\in\R_+$ is linked to $\rho$ by a compatibility equation to be defined later.
\end{enumerate}
\end{theorem}

The first point of this theorem is detailed in Section \ref{paragraphequilibria} (see in particular Theorem \ref{solutionscompatibility} and Corollary \ref{phasediagramcorollary}). The second point is detailed in Section \ref{convergence} (see in particular Theorem \ref{theoremconvergence}). We will then prove the following result. 

\begin{theorem}[Formal]\label{theorem2} Let us consider the rescaled spatially inhomogeneous problem
\[\partial_t f^\varepsilon+(Ae_1\cdot \nabla_x)f^\varepsilon=\frac{1}{\varepsilon}\Big(\rho_{f^\varepsilon} M_{J_{f^\varepsilon}}-f^\varepsilon\Big),\]
where 
\[\rho_f(t,x) = \int_{SO_3(\R)} f(t,x,A)\,dA\,\,\,\,\,\text{and}\,\,\,\,\,J_f(t,x)=\int_{SO_3(\R)} f(t,x,A)A\,dA.\]
\begin{enumerate}
\item We assume that in a disordered region, $f^\varepsilon$ converges as $\varepsilon\to0$ towards a density $\rho=\rho(t,x)$ uniform in the body-attitude variable. Then the density $\rho^\varepsilon\equiv\rho_{f^\varepsilon}$ satisfies at first order the following diffusion equation: 
\[\partial_t\rho^\varepsilon=\varepsilon\nabla_x\cdot\left(\frac{\frac{1}{3}\nabla_x\rho^\varepsilon}{1-\frac{\rho^\varepsilon}{\rho_c}}\right),\,\,\,\,\,\,\rho_c=6.\]
\item We assume that in an ordered region, $f^\varepsilon$ converges as $\varepsilon\to0$ towards an equilibrium of the form $\rho M_{\alpha\Lambda}$ with $\rho\in\R_+$, $\alpha\in\R_+$ and $\Lambda\in SO_3(\R)$ defined above. Then the density $\rho=\rho(t,x)$ and mean body attitude $\Lambda=\Lambda(t,x)\in SO_3(\R)$ satisfy the SOHB model given by the following system of partial differential equations:
\begin{subequations}
\label{SOHB}
\begin{align}
&\partial_t \rho+\nabla_x\cdot (\rho c_1(\alpha(\rho))\Lambda e_1)=0, \label{SOHB1} \\
&\rho(\partial_t\Lambda+\tilde{c}_2((\Lambda e_1)\cdot\nabla_x)\Lambda)+\tilde{c_3}[(\Lambda e_1)\times\nabla_x\rho]_\times\Lambda \nonumber \\ 
&\hspace{4.2cm}+c_4\rho[-\mathbf{r}_x(\Lambda)\times(\Lambda e_1)+\delta_x(\Lambda)\Lambda e_1]_\times\Lambda=0. \label{SOHB2}
\end{align}
\end{subequations}
where $\alpha=\rho c_1(\alpha)$ and $\tilde{c}_2$, $\tilde{c_3}$, $c_4$ are functions of $\rho$ to be defined later and $\delta$ and $\mathbf{r}$ are the ``divergence'' and ``rotational'' operators defined in \cite{degondfrouvellemerino17} (see Section \ref{macroscopic})
\end{enumerate}
\end{theorem}
This theorem is detailed in Section \ref{macroscopic} (see in particular Proposition \ref{diffusionproposition} and Theorem \ref{theoremSOHB}).\\

The phase transition problem has been completely treated in the space-homogeneous case for the Vicsek model in \cite{degondfrouvelleliu15} but the geometrical structure inherent to body-orientation models requires specific tools and techniques. In particular, the rotation group $SO_3(\R)$ is a compact Lie group, endowed with a Haar measure. The links between this topological structure and the Riemannian structure (detailed in Section \ref{SO3technique} and Appendix \ref{moreonson}) will be the key to
reduce the problem to a form that shares structural properties with the models of nematic alignment of polymers, studied in a completely different context to model liquid crystals \cite{han2015microscopic,wanghoffman08,wangzhou11,ball2010nematic,ball2017mathematics}. These two worlds will be formally linked through the isomorphism between $SO_3(\R)$ and the group of unit quaternions detailed in Section \ref{paragraphdiagonalsolutions} and Appendix \ref{quaternionsandrotations}. It will lead to the first point of Theorem \ref{theorem1} (the complete description of the equilibria, Section \ref{paragraphequilibria}). As in \cite{wanghoffman08} we will see that there exist a class of equilibria which cannot be interpreted as equilibria around a mean-body orientation. These equilibria were not studied in \cite{degondfrouvellemerino17,degondfrouvellemerinotrescases18}. A key point of the proof will be the reduction to a problem for diagonal matrices which will be a consequence of the left and right invariance of the Haar measure together with an adapted version of the Singular Value Decomposition of a matrix (Definition \ref{ssvd}). \\

The stability of the different equilibria, are studied in Section \ref{paragraphstability}. We will show that our model has an underlying gradient-flow structure which will allow us to determine the asymptotic behaviour of the system after a reduction to an ODE in $\R^3$. This is a specificity of the BGK model which doesn't hold for the other models of body-attitude coordination \cite{degondfrouvellemerino17,degondfrouvellemerinotrescases18} and allows us to use different and simpler techniques. In particular, we will prove that the equilibria which cannot be interpreted as equilibria around a mean body-orientation are always unstable, which tends to justify the analysis carried out in \cite{degondfrouvellemerino17} for a model where only equilibria around a mean body-orientation were considered. \\

Finally, the SOHB model (Section \ref{macroscopic}) will be obtained as in \cite{degondfrouvellemerino17} by using the GCI. However, compared to \cite{degondfrouvellemerino17}, additional terms appear which require a specific treatment and in particular the coefficient $\tilde{c_3}$ that appears in Theorem \ref{theorem2} is different from the one that appears in \cite{degondfrouvellemerino17}. The SOHB model \eqref{SOHB} raises many questions, most of which are still open, and its mathematical and numerical analyses are still in progress. In particular, the hyperbolicity of the model is currently under study \cite{hyperbolicityinprogress} and has been shown when $\tilde{c_3}$ is constant. \\

The organisation of the work is the following: in Section \ref{model} we will give a review of the existing models at a microscopic and mesoscopic scales and motivate the study of the BGK equation among them. In Section \ref{SO3technique}, we gather the main technical results we will constantly use throughout this work. In Section \ref{paragraphequilibria}, we will describe, depending on the density, all the possible equilibria of the system. We will use the tools developed to mathematically study the alignment of polymers \cite{wanghoffman08,wangzhou11}. In Section \ref{convergence} we will describe the asymptotic behaviour of the system and in particular which equilibria are attained, leading to a self-organised dynamics or not. This will be based on a specific underlying gradient-flow structure of the BGK equation. Finally in Section \ref{macroscopic} we will write the macroscopic models for the stable equilibria. \\ 

\noindent\textbf{Notations.} For the convenience of the reader, we collect here the main notations we will use in the following.
\begin{enumerate}[$\bullet$]
\item $\MM_n(\R)$ is the set of $n\times n$ real matrices.
\item $\mathscr{D}_n(\R)\subset\MM_n(\R)$ is the subspace of $n\times n$ diagonal real matrices.
\item $\Tr(M)$ denotes the trace of the matrix $M\in\MM_n(\R)$ and $M^T$ its transpose.
\item $I_n$ denotes the identity matrix in dimension $n$.
\item $\diag:\R^n\to\mathscr{D}_n(\R)$ is the vector space isomorphism such that for $(d_1,\dots,d_n)\in\R^n$, $D=\diag(d_1,\dots,d_n)$ is the diagonal matrix, the $(i,i)$-th coefficient of which is equal to $d_i$ for $i\in\{1,\ldots,n\}$.
\item $\mathscr{S}_n(\R)$ and $\mathscr{A}_n(\R)$ denote respectively the sets of symmetric and skew-symmetric matrices of dimension $n$. 
\item $SO_n(\R)$ is the special orthogonal group in dimension $n$, i.e. the group of matrices $P\in\MM_n(\R)$ such that $PP^T=I_n$ and $\det P>0$.
\item $\Sph^n\subset\R^{n+1}$ is the sphere of dimension $n$.
\item $\HH$ is the group of unitary quaternions.
\item $\langle\cdot\rangle_{g}$ denotes the mean for the probability density $g$ on $SO_n(\R)$. We will simply write $\langle\cdot\rangle$ when $g$ is the uniform probability ($g\equiv1$).
\item $A$ will generically be a rotation matrix in $SO_3(\R)$ and $a_{ij}$ its $(i,j)$ coefficient. 
\item $PD(M)$ is the orthogonal part of the polar decomposition of $M\in\MM_n(\R)$ when $\det M\ne0$~: there exists a unique couple $(PD(M),S)\in\OO_n(\R)\times\SS_n(\R)$ such that $M=PD(M)S$. The matrix $PD(M)$ is given by $PD(M)=M\left(\sqrt{M^TM}\right)^{-1}$. 
\item For a matrix $M\in\MM_3(\R)$, the \textit{orbit} $\Orb(M)\subset \MM_3(\R)$ is defined by: 
\begin{equation}\label{orbit}\Orb(M):=\{PMQ,\,\,\,P,Q\in SO_3(\R)\}.\end{equation}
\item $\R_+:=[0,+\infty)$, $\R_+^*:=(0,+\infty)$
\end{enumerate}

\section{The BGK equation and other related models of self-organisation}\label{model}

In this section we give a review of the different existing models of collective dynamics at both microscopic and mesoscopic levels and emphasise the singularity of the BGK model among them.

\subsection{A review of the different IBM}

The rigorous proofs of the two following theorems (Theorems \ref{theoremmeanfieldinhomogeneous} and \ref{theoremmeanfieldhomogeneous}) can be found in \cite{diez} in a more general framework. \\

At a microscopic level, we fix a reference frame given by the canonical basis $(e_1,e_2,e_3)$ of $\R^3$. The agents are described by their position $X\in\R^3$ and their body-attitude $A\in SO_3(\R)$ which can be seen as a moving frame. We assume that an agent with body attitude $A\in SO_3(\R)$ moves at a constant speed in the direction of the first vector of $A$~: the instantaneous velocity of the agent is $A e_1$.\\

In the following we consider an increasing sequence of jump times $(T_n)_n$ such that the increments between two jumps are independent and follow an exponential law of parameter $N\in\N^*$ (their expectation is $1/N$). The $N$ agents are described at time $t\in\R_+$ by their positions and body-attitudes $Z^{N}_t=\big\{(X^{i,N}_t,A^{i,N}_t)\big\}_{i\in\{1,\ldots,N\}}\in \big(\R^3\times SO_3(\R)\big)^N$. The interactions between the agents can be modelled by the following Piecewise Deterministic Markov Process (PDMP) which has already been described heuristically in \cite{degondfrouvellemerinotrescasesAlignment18,dimarcomotsch16}: 
\begin{enumerate}
\item Between two jump times $({T}_n,{T}_{n+1})$, the systems evolves in a deterministic way:
\[\forall i\in\{1,\ldots,N\},\,\,\,\,\,\,\,dX_t^{i,N}=(A^{i,N}_te_1)dt,\,\,\,\,\,\,\,dA^{i,N}_t=0.\]
\item At time ${T}_{n+1}$, a particle $i\in\{1,\ldots,N\}$ is chosen uniformly among the $N$ particles.  At time ${T}_{n+1}^+$, the new body-orientation of particle $i$ is sampled from the PDF $M_{J^i\big(Z_{{T}_{n+1}^-}^N\big)}$ where for $Z^N\in(\R^3\times SO_3(\R))^N$ we define the flux:
\[J^{i}(Z^N):=\frac{1}{N}\sum_{j=1}^{N} K\Big(|X^{i,N}-X^{j,N}|\Big)A^{j,N}\in\MM_3(\R),\]
and where $K$ is a smooth observation kernel.
\end{enumerate}
The following theorem describes the limiting behaviour of the laws of the particles when $N\to+\infty$ under the assumption that the empirical measure of the $N$ processes converges (weakly) towards a smooth function $f$ (propagation of chaos property). The equation on $f$ can be derived as in \cite[Section 4.2]{degondfrouvellemerinotrescasesAlignment18}. 
\begin{theorem}\label{theoremmeanfieldinhomogeneous}
Let $f_0$ be a probability measure on the space $\R^3\times SO_3(\R)$ and let $Z_0^N\in(\R^3\times SO_3(\R))^N$ be an initial state given by $N$ independent random variables, identically distributed with law $f_0$. Then for any $t\in\R_+$, the law $f^N_t$ of any of one of the processes $(X^{i,N}_t,A^{i,N}_t)_t$ at time $t$ converges weakly towards the solution $f_t$ of the following BGK equation with initial condition $f_0$: 
\[\partial_t f+(Ae_1\cdot\nabla_x)f=\rho_f M_{J_{K*f}}-f\]
where $J_{K*f}$ is a matrix-valued function of the space variable $x\in\R^3$ defined by
\[J_{K*f}(x):=\iint_{\R^3\times SO_3(\R)} K(x-y)Af(y,A)dy dA\in\MM_3(\R).\]
\end{theorem}
In the previous works, the IBM were typically given as in \cite{degondfrouvellemerino17} by a system of stochastic differential equations such as the following: 
\begin{subequations}
\begin{align}
&{d X_k}=A_k(t)e_1\,dt,\\
& {dA_k}= P_{T_{A_k}}\circ \left(\Big(\frac{1}{N}\sum_{i=1}^{N} K(|X_i-X_k|) A_i\Big)dt+2\sqrt{D}dB_t\right).
\end{align}
\end{subequations}
where $P_{T_{A_k}}$ denotes the projection on the tangent space of $SO_3(\R)$ at $A_k\in SO_3(\R)$ (see Section \ref{SO3technique}). In this case, the resulting equation when $N\to+\infty$ is a non-linear Fokker-Planck equation (see \cite{bolley2012mean} for a rigorous proof in the Vicsek case).\\

In the spatially homogeneous case, we can take the observation kernel $K$ to be constantly equal to 1 to prove the mean-field limit. The agents are described at time $t\in\R_+$ only by their body-attitudes $\big\{A^{i,N}\big\}_{i\in\{1,\ldots,N\}}\in SO_3(\R)^N$ and they follow the following jump process: at each jump time $T_n$, compute the flux
\[J^N_t=\frac{1}{N}\sum_{i=1}^N A_t^{i,N},\]
choose a particle $i\in\{1,\ldots N\}$ uniformly among the $N$ particles and draw the new body-orientation $A^{i,N}_{T_n^+}$ after the jump according to the law given by the PDF $M_{J^N_{T_n^-}}$.
The following theorem describes analogously the limiting behaviour of the laws of the particles as $N\to+\infty$.
\begin{theorem}\label{theoremmeanfieldhomogeneous}
Let $\big\{A^{i,N}_0\big\}_{i\in\{1,\ldots,N\}}\in SO_3(\R)^N$ be an initial state given by $N$ independent random variables, identically distributed according to a law $f_0$ on $SO_3(\R)$. Then for any $t\in\R_+$, the law $f^N_t$ of any of one of the processes $(A^{i,N}_t)_t$ at time $t$ converges weakly towards the solution $f_t$ of the following spatially homogeneous BGK equation with initial condition $f_0$: 
\[\partial_t f=M_{J_{f}}-f.\]
\end{theorem}

\subsection{A review of the different kinetic equations} 

The model studied in the present article belongs to a class of models, the study of which has been initiated in \cite{degondmotsch08} as a continuum version of the Vicsek model \cite{vicsek1995novel}. These models can be classified in two types. First, in the Vicsek-type models, the agents are described by their orientation defined as a unit vector in $\Sph^{n-1}$. In the second type of models, we take into account their body-orientation, defined as a rotation matrix in $SO_3(\R)$. Our study enters into this second framework.  \\

The kinetic version of the Vicsek-type or Body-Orientation-type models is given either by a Fokker-Planck equation or by a BGK equation. In this work we will focus on the BGK equation 
\begin{equation}\label{bgk}
\partial_t f+(Ae_1\cdot\nabla_x)f=\rho_f M_{J_f}-f,\end{equation}
where
\[\rho_f(t,x)=\int_{SO_3(\R)} f(t,x,A)\,dA\,\,\,\,\,\,\text{and}\,\,\,\,\,J_f=\int_{SO_3(\R)} f(t,x,A)A\,dA.\]
The Fokker-Planck version of our model corresponds to: 
\begin{equation}\label{FP}\partial_t f +(Ae_1\cdot \nabla_x)f=\nabla_A\cdot\left[M_{J_f}\nabla_A\left(\frac{f}{M_{J_f}}\right)\right],\end{equation}
where $\nabla_A$ and $\nabla_A\cdot$ are respectively the gradient and the divergence in $SO_3(\R)$ for the Riemannian structure detailed in Section \ref{SO3technique}. Apart from the fact that the underlying interaction process \cite{dimarcomotsch16,degondfrouvellemerinotrescasesAlignment18} which leads to the BGK model is different from the one that leads to the Fokker-Planck model, the BGK model is structurally different and can be treated independently by using specific and simpler mathematical techniques presented in the next sections. Nevertheless, the BGK and Fokker-Planck models share important properties. For instance, the following functional is a free-energy for both the spatially homogeneous BGK equation and the spatially homogeneous Fokker-Planck equation (though with a different dissipation term): 
\begin{equation}\label{freeenergy}\mathcal{F}[f]:=\int_{SO_3(\R)} f\log f-\frac{1}{2}|J_f|^2.\end{equation}
It satisfies in both cases: 
\[\frac{d}{dt}\mathcal{F}[f]=-\mathcal{D}[f]\leq0,\]
where $\mathcal{D}[f]$ is the dissipation term which is equal for the BGK model to: 
\[\mathcal{D}[f]=\int_{SO_3(\R)} (f-\rho M_{J_f})(\log f-\log(\rho M_{J_f}))\geq0.\]
In the context of the Vicsek model, this free energy was the key to study the phase transition phenomena \cite{degondfrouvelleliu15} and we believe that the same kind of study can be made in the body-attitude coordination dynamics modelled by a Fokker-Planck equation \eqref{FP}. Moreover, in the Fokker-Planck case this dissipation inequality implies a gradient flow structure in the Wasserstein-2 distance which has been studied (in the Vicsek case) in \cite{figalli2018global}. However, the BGK model has another underlying gradient-flow dynamics (studied in Section \ref{convergence}) on which the present study will be based, and we will therefore not use this free-energy in the present work.\\

Both models (BGK and Fokker-Planck) have a normalised and a non-normalised version. The model \eqref{bgk} will be referred as the non-normalised BGK model. A normalised model is a model where the flux $J_f$ is replaced by the orthogonal part of its polar decomposition $\Lambda_f:=PD(J_f)$ as defined in the introduction and under the assumption that $\det J_f>0$. The normalised Fokker-Planck model is the model studied in \cite{degondfrouvellemerino17}:
\[\partial_t f +(Ae_1\cdot \nabla_x)f=\nabla_A\cdot\left[M_{\Lambda_f}\nabla_A\left(\frac{f}{M_{\Lambda_f}}\right)\right].\]
This terminology comes from the continuum version of the Vicsek model \cite{degondmotsch08} where either the total flux
\[J_f=\int_{\Sph^{n-1}}f(t,x,v)v\,dv,\]
or its normalisation 
\[\Omega_f=\frac{\int_{\Sph^{n-1}} f(t,x,v)v\,dv}{\left|\int_{\Sph^{n-1}} f(t,x,v)v\,dv\right|}\in\Sph^{n-1},\]
is considered. A mathematical analysis of the normalised Vicsek model can be found in \cite{figalli2018global,gamba2016global}. The importance of this distinction in the context of phase transitions has been shown in \cite{degondfrouvelleliu15} and \cite{degondfrouvelleliu13}: phase transitions appear only in non-normalised models. \\

The following chart (Figure \ref{models}) shows the different models and gives references where they are studied (when such references exist).

\begin{figure}[H]
\begin{center}
\begin{tikzpicture}[xscale=1,yscale=0.9]
\tikzstyle{fleche}=[->,>=latex,thick]
\tikzstyle{noeud}=[circle,draw]
\tikzstyle{feuille}=[rectangle,draw]
\tikzstyle{etiquette}=[midway,fill=white]
\def\DistanceInterNiveaux{3}
\def\DistanceInterFeuilles{1.5}
\def\NiveauA{(0)*\DistanceInterNiveaux}
\def\NiveauB{(1.6666666666666665)*\DistanceInterNiveaux}
\def\NiveauC{(3)*\DistanceInterNiveaux}
\def\NiveauD{(4)*\DistanceInterNiveaux}
\def\InterFeuilles{(-1)*\DistanceInterFeuilles}
\node[noeud] (R) at ({\NiveauA},{(3.5)*\InterFeuilles}) {Model};
\node (Ra) at ({\NiveauB},{(1.5)*\InterFeuilles}) {};
\node (Raa) at ({\NiveauC},{(0.5)*\InterFeuilles}) {};
\node[right] (Raaa) at ({\NiveauD},{(0)*\InterFeuilles}) {\cite{degondmotsch08,figalli2018global,gamba2016global}};
\node[right] (Raab) at ({\NiveauD},{(1)*\InterFeuilles}) {\cite{bolley2012mean,degondfrouvelleliu15,degondfrouvelleliu13}};
\node (Rab) at ({\NiveauC},{(2.5)*\InterFeuilles}) {};
\node[right] (Raba) at ({\NiveauD},{(2)*\InterFeuilles}) {\cite{dimarcomotsch16}};
\node[right] (Rabb) at ({\NiveauD},{(3)*\InterFeuilles}) {\textit{In progress}};
\node (Rb) at ({\NiveauB},{(5.5)*\InterFeuilles}) {};
\node (Rba) at ({\NiveauC},{(4.5)*\InterFeuilles}) {};
\node[right] (Rbaa) at ({\NiveauD},{(4)*\InterFeuilles}) {\cite{degondfrouvellemerino17,degondfrouvellemerinotrescasesAlignment18}};
\node[right] (Rbab) at ({\NiveauD},{(5)*\InterFeuilles}) {\textit{In progress}};
\node (Rbb) at ({\NiveauC},{(6.5)*\InterFeuilles}) {};
\node[right] (Rbba) at ({\NiveauD},{(6)*\InterFeuilles}) {\cite{degondfrouvellemerinotrescasesAlignment18}};
\node[right] (Rbbb) at ({\NiveauD},{(7)*\InterFeuilles}) {\textbf{Present work}};
\draw[fleche] (R)--(Ra) node[etiquette] {Vicsek};
\draw[fleche] (Ra)--(Raa) node[etiquette] {Fokker-Planck};
\draw[fleche] (Raa)--(Raaa) node[etiquette] {\scriptsize Normalised};
\draw[fleche] (Raa)--(Raab) node[etiquette] {\scriptsize Non-normalised};
\draw[fleche] (Ra)--(Rab) node[etiquette] {BGK};
\draw[fleche] (Rab)--(Raba) node[etiquette] {\scriptsize Normalised};
\draw[fleche] (Rab)--(Rabb) node[etiquette] {\scriptsize Non-normalised};
\draw[fleche] (R)--(Rb) node[etiquette] {Body-Orientation};
\draw[fleche] (Rb)--(Rba) node[etiquette] {Fokker-Planck};
\draw[fleche] (Rba)--(Rbaa) node[etiquette] {\scriptsize Normalised};
\draw[fleche] (Rba)--(Rbab) node[etiquette] {\scriptsize Non-normalised};
\draw[fleche] (Rb)--(Rbb) node[etiquette] {BGK};
\draw[fleche] (Rbb)--(Rbba) node[etiquette] {\scriptsize Normalised};
\draw[fleche] (Rbb)--(Rbbb) node[etiquette] {\scriptsize Non-normalised};
\end{tikzpicture}
\end{center}
\caption{The map of the different models}
\label{models}
\end{figure}
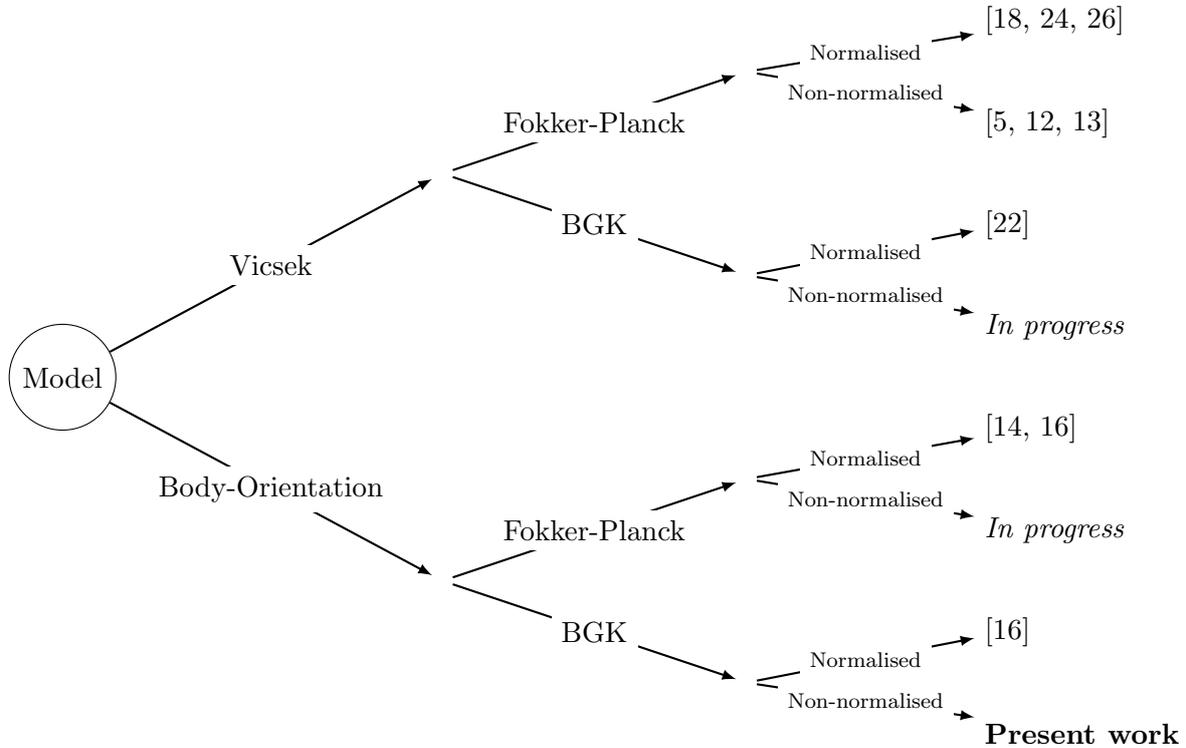

Finally, in Sections \ref{paragraphequilibria} and \ref{convergence}, we will focus on the spatially homogeneous version of the BGK model \eqref{bgk} given by:
\begin{equation}\label{bgkh}\partial_t f = \rho{M_{J_f}}-f,\end{equation}
where the probability distribution $f(t,A)$ only depends on the body-orientation variable and time. In the spatially homogeneous case, the local density of agents previously denoted by $\rho_f$ does not depend on $f$ in the sense that an initial density $\rho_{f_0}\in\R_+$ associated to the initial distribution $f_0$ is preserved by the dynamics: 
\[\forall t\in\R_+,\,\,\,\,\, \rho_f(t)=\rho_{f_0},\]
as it can be seen by integrating the equation over $SO_3(\R)$. We therefore take $\rho\in\R_+$ as a fixed parameter of the problem. Note also that the well-posedness of \eqref{bgkh} directly follows from Duhamel's formula: 
\[f(t)=e^{-t}f_0+\rho\int_0^t e^{-(t-s)}M_{J_{f(s)}}\,ds,\]
since $J_f$ is given as the solution of the following differential equation on $\MM_3(\R)$: 
\[\frac{d}{dt}J_f=\rho\langle A\rangle_{M_{J_f}}-J_f,\,\,\,\,\,\,J_f(t=0)=J_{f_0},\]
as it can be seen by multiplying \eqref{bgkh} by $A$ and integrating over $SO_3(\R)$. Note that it contrasts with the Fokker-Planck case where even the well-posedness of the spatially-homogeneous equation would require further investigations. This will be part of future work. 

\section{Preliminaries: structure and calculus in $SO_n(\R)$}\label{SO3technique}

This paragraph collects the main properties of the Riemannian manifold $SO_n(\R)$ and other technical results. In this paragraph $n\geq3$ denotes the dimension, we will mainly consider the case $n=3$ in the next sections.

\subsection{Structure and Haar measure on $SO_n(\R)$}
\begin{lemma}
The following is an inner product on $\MM_n(\R)$~: 
\begin{equation}\label{produitscalaire}A\cdot B:=\frac{1}{2}\Tr(A^T B),\end{equation}
and the following properties hold:
\begin{enumerate}[$\bullet$]
\item Endowed with this metric, $SO_n(\R)$ is a topological group and a Riemannian manifold.
\item The sets $\mathscr{S}_n(\R)$ and $\mathscr{A}_n(\R)$ of symmetric and skew-symmetric matrices are orthogonal and $\MM_n(\R)=\mathscr{S}_n(\R)\oplus\mathscr{A}_n(\R)$.
\item For $A\in SO_n(\R)$, the tangent space to $SO_n(\R)$ at $A$ is denoted by $T_A$ and 
\[M\in T_A\,\,\,\,\,\text{if and only if there exists}\,\,\,P\in\mathscr{A}_n(\R)\,\,\,\text{such that}\,\,\,M=AP.\]
\end{enumerate}
The norm on $\MM_n(\R)$ associated to the inner product \eqref{produitscalaire} will be denoted by $\|\cdot\|$.
\end{lemma}
The general theory of locally compact topological groups ensures the existence of a Haar measure $\mu$ on $SO_n(\R)$ which satisfies for all $P\in SO_n(\R)$ and all Borel set $\mathcal{E}$ of the Borel $\sigma$-algebra of $SO_n(\R)$: 
\[\mu(P\mathcal{E})=\mu(\mathcal{E}P)=\mu(\mathcal{E}),\]
where $P\mathcal{E}=\{PA,\,\,A\in\mathcal{E}\}$ and $\mathcal{E}P=\{AP,\,\,A\in\mathcal{E}\}$. We will assume that $\mu$ is the unique Haar measure which is a probability measure and simply write 
\[\int_{SO_n(\R)} f(A)\,d\mu(A)\equiv\int_{SO_n(\R)} f(A)\,dA.\]
As a consequence if $P\in SO_n(\R)$, $A\mapsto PA$ and $A\mapsto AP$ are two changes of variable with unit Jacobian. We will constantly use the following changes of variable : 
\begin{definition}[Useful changes of variable]\label{changeofvariable}
Let us define the following matrices:
\begin{enumerate}[$\bullet$]
\item For $i\ne j\in\{1,\ldots,n\}$, $D^{ij}\in SO_n(\R)$ is the diagonal matrix such that all its coefficients are equal to $1$ except at positions $i$ and $j$ where they are equal to $-1$. 
\item For $i\ne j\in\{1,\ldots,n\}$, $P^{ij}\in SO_n(\R)$ is the matrix such that $P^{ij}_{ii}=P^{ij}_{jj}=0$, $P^{ij}_{kk}=1$ for $k\ne i,j$, $P^{ij}_{ij}=1$ and $P^{ij}_{ji}=-1$. The other coefficients are equal to 0.
\end{enumerate}
Then we define the following changes of variable with unit Jacobian:
\begin{enumerate}[$\bullet$]
\item $A'=D^{ij}A$ multiplies the rows $i$ and $j$ by $-1$. Everything else remains unchanged.
\item $A'=AD^{ij}$ multiplies the columns $i$ and $j$ by $-1$. Everything else remains unchanged. 
\item $A'=D^{ij} AD^{ij}$ multiplies the elements $(k,i)$, $(k,j)$ and $(i,k)$, $(j,k)$ by $-1$ for $k\ne i,j$. Everything else remains unchanged
\item $A'=P^{ij}A$ multiplies row $i$ by $-1$ and permutes the rows $i$ and $j$. 
\item $A'=P^{ij} A(P^{ij})^T$ exchanges the diagonal coefficients $(i,i)$ and $(j,j)$ (and involves other changes). 
\end{enumerate}
\end{definition}
The two following lemmas are important applications of these results.
\begin{lemma}\label{diagonal}
Let $D\in\MM_n(\R)$ be a diagonal matrix and $M_D$ the von Mises distribution with parameter $D$, then
\[\langle A\rangle_{M_D}:=\int_{SO_3(\R)}\,AM_D(A)\,dA\]
is diagonal.
\end{lemma}

\begin{proof}
Let $k\ne\ell$ and $m\ne k,\ell$. The change of variable $A\mapsto D^{km}AD^{km}$ gives: 
\[\int_{SO_n(\R)} a_{k,\ell}\,e^{D\cdot A}\,dA=-\int_{SO_3(\R)} a_{k,\ell}\,e^{D\cdot A}\,dA=0,\]
where we have used that $D^{k,m}DD^{k,m}=D$.
\end{proof}

\begin{lemma}\label{JAA} For any $n\geq3$ and any $J\in\MM_n(\R)$, 
\[\int_{SO_n(\R)} (J\cdot A)A\,dA= \frac{1}{2n} J.\]
\end{lemma}

\begin{lemma}\label{JAAg}
Let $n\geq3$, $n\ne4$. Let $g:SO_n(\R)\to\R$ such that for all $A,P\in SO_n(\R)$, $g(A)=g(A^T)~=~g(PAP^T)$. For all $J\in\MM_n(\R)$ we have: 
\[\int_{SO_n(\R)} (J\cdot A)A\,g(A)\,dA= a\Tr(J) I_n+b J+cJ^T,\]
for given $a, b, c\in\R$ depending on $g$ and on the dimension, the expressions of which can be found in the proof.
\end{lemma}

The proof of these lemmas and other technical results about $SO_3(\R)$ and $SO_n(\R)$ are postponed to Appendix \ref{moreonson}.

\subsection{Volume forms in $SO_3(\R)$}\label{volumeforms}

When an explicit calculation will be needed, we will use one of the two following parametrisations of $SO_3(\R)$ which give two explicit expressions of the normalised Haar measure in dimension 3. 
\begin{enumerate}[$\bullet$]
\item To a matrix $A\in SO_3(\R)$ there is an associated angle $\theta\in[0,\pi]$ and a vector $\mathbf{n}\in \Sph^2$ such that $A$ is the rotation of angle $\theta$ around the axis $\mathbf{n}$. Rodrigues' formula gives a representation of $A$ knowing $\theta$ and $\mathbf{n}=(n_1,n_2,n_3)$~: 
\begin{equation}\label{rodrigue}
A=A(\theta,\mathbf{n})=I_3+ \sin\theta[\mathbf{n}]_\times+(1-\cos\theta)[\mathbf{n}]_\times^2=\exp(\theta[\mathbf{n}]_\times),
\end{equation}
where 
\[[\mathbf{n}]_\times:=\left(\begin{array}{ccc}
 0 & -n_3 & n_2 \\ 
 n_3 &0 & -n_1 \\ 
 -n_2 & n_1 & 0
 \end{array}\right),\]
 and we have: 
 \[[\mathbf{n}]_\times^2 = \mathbf{n}\otimes\mathbf{n}-I_3.\]
If $f(A(\theta,\mathbf{n}))=\bar{f}(\theta,\mathbf{n})$ the volume form of $SO_3(\R)$ is given by: 
\[\int_{SO_3(\R)} f(A)\,dA= \frac{2}{\pi}\int_0^\pi \sin^2(\theta/2)\int_{\Sph^2} \bar{f}(\theta,\mathbf{n})\,d\mathbf{n}\,d\theta.\]
With the usual parametrisation of the sphere $\Sph^2$ we can take $\mathbf{n}=(n_1,n_2,n_3)^T$ with 
\[\left\{\begin{array}{rcl}
n_1&=&\sin\psi\cos\varphi, \\
n_2&=&\sin\psi\sin\varphi, \\ 
n_3&=& \cos\psi,
\end{array}\right.
\]
where $\psi\in[0,\pi]$ and $\varphi\in[0,2\pi]$. The volume form for the sphere is given by: 
\[d\mathbf{n}=\frac{1}{4\pi}\sin\psi d\psi d\varphi.\]

\item We have the following one to one map : 
\begin{equation}\label{SOSphere}\Psi :\left|\begin{array}{rcl}
 SO_{2}(\R)\times \Sph^{2}&\longrightarrow& SO_3(\R)\\
 (A,p)&\longmapsto& M(p)A^a
 \end{array}\right.
 \end{equation}
where
\[A^a:=\left(\begin{array}{cc}A & 0 \\ 0 & 1\end{array}\right)\in SO_3(\R),\] 
and for $p=(\sin\phi_{1}\,\sin\phi_{2},\cos\phi_{1}\,\sin\phi_{2},\cos{\phi_2})^T$ in spherical coordinates $\phi_{1}\in[0,2\pi]$ and $\phi_2\in[0,\pi]$, we define:
\[M(p):=\left(\begin{array}{ccc}
\cos\phi_{1} & \sin\phi_{1}\,\cos\phi_{2} & \sin\phi_{1}\,\sin\phi_{2} \\ 
-\sin\phi_{1} & \cos\phi_{1}\,\cos\phi_{2} & \cos\phi_{1}\,\sin\phi_{2} \\ 
0 & -\sin\phi_{2} & \cos{\phi_2}\end{array}\right)\in SO_3(\R).
\]
The matrix $A^a$ performs an arbitrary rotation of the first $2$ coordinates and the matrix $M(p)\in SO_3(\R)$ maps the vector $e_3$ to $p\in\Sph^{2}$. A matrix $A\in SO_3(\R)$ can thus be written as the product: 
\[\left(\begin{array}{ccc}
\cos\phi_{1} & \sin\phi_{1}\,\cos\phi_{2} & \sin\phi_{1}\,\sin\phi_{2} \\ 
-\sin\phi_{1} & \cos\phi_{1}\,\cos\phi_{2} & \cos\phi_{1}\,\sin\phi_{2} \\ 
0 & -\sin\phi_{2} & \cos{\phi_2}\end{array}\right)
\left(\begin{array}{ccc}
\cos\theta & \sin\theta & 0 \\ 
-\sin\theta & \cos\theta & 0 \\ 
0 & 0 & 1
\end{array}\right)\]
where $\phi_{1},\theta\in[0,2\pi]$ and $\phi_2\in[0,\pi]$. With this parametrisation: 
\begin{equation}\label{parametrisationSOSphere}\int_{SO_3(\R)}f(A)dA= \frac{1}{2\pi}\int_{0}^{2\pi} \int_{\Sph^2} f\left(M(p)\left(\begin{array}{ccc}
\cos\theta & \sin\theta & 0 \\ 
-\sin\theta & \cos\theta & 0 \\ 
0 & 0 & 1
\end{array}\right)\right)\,d\theta\,dp,\end{equation}
and the volume form on the sphere is given by: 
\[dp=\frac{1}{4\pi}\sin\phi_2\,d\phi_1\,d\phi_2.\]
\end{enumerate}
This parametrisation can be extended in any dimension and comes from the Lie groups quotient:
\[\frac{SO_n(\R)}{SO_{n-1}(\R)}\cong\Sph^{n-1}.\]

\subsection{Singular Value Decomposition (SVD)}

We recall the following classical result proved in \cite[Section 1.9]{quarteroni2010numerical}.
\begin{proposition}[Singular Value Decomposition, SVD]\label{svd} Any square matrix $M\in\MM_n(\R)$ can be written: 
\[M=PDQ\]
where $P,Q\in \OO_n(\R)$ and $D$ diagonal with nonnegative coefficients listed in decreasing order.
\end{proposition} 

In order to use the properties of the Haar measure, we will need the matrices $P$ and $Q$ to belong to $SO_3(\R)$ (not only $\OO_3(\R))$ and we define therefore another decomposition, called the Special Singular Value Decomposition (SSVD) in the following.

\begin{definition}[SSVD in $SO_3(\R)$]\label{ssvd} Let $M\in\MM_3(\R)$. A \textit{Special Singular Value Decomposition} (SSVD) of $M$ is a decomposition of the form 
\[M=PDQ\]
where $P,Q\in SO_3(\R)$ and $D=\diag(d_1,d_2,d_3)$ with 
\[d_1\geq d_2\geq|d_3|.\]
\end{definition}
The existence of a SSVD follows from Proposition \ref{svd}. Let us start from a SVD
\[M=P'D'Q'.\]
\begin{enumerate}[$\bullet$]
\item If $\det M>0$, either $P',Q'\in SO_3(\R)$ and the SVD is a SSVD or $P',Q'$ have both negative determinant and in this case we can take 
\[P=P'\tilde{D},\,\,\,\,Q=\tilde{D}Q'\,\,\,\,\text{and}\,\,\,\,D=D'\]
where $\tilde{D}=\diag(1,1,-1)$.
\item If $\det M<0$, either $P'\in SO_3(\R)$ or $Q'\in SO_3(\R)$ (only one of them). Assume without loss of generality that $Q'\in SO_3(\R)$. Then we can take: 
\[P=P'\tilde{D},\,\,\,\, D=\tilde{D}D\,\,\,\,\text{and}\,\,\,\,Q=Q'.\]
\item If $\det M=0$, then the last coefficient of $D'$ is equal to 0 so $\tilde{D}D'=D'$ and $D'\tilde{D}=D'$. We can take $D=D'$. If $P'\notin SO_3(\R)$ we can take $P=P'\tilde{D}$ and if $Q'\notin SO_3(\R)$ we can take $Q=\tilde{D}Q'$.
\end{enumerate}

\begin{rem} As for the polar decomposition and the standard SVD, the matrix $D$ is always unique. However the matrices $P$ and $Q$ may not be unique.
\end{rem}

The subset $\mathscr{D}\subset\MM_3(\R)$ of the diagonal matrices which are the diagonal part of a SSVD is the cone delimited by the image by the isomorphism $\diag$ of the three planes $\{d_1=d_2\}$, $\{d_2=d_3\}$ and $\{d_2=-d_3\}$ in $\R^3$ and depicted in Figure \ref{dessinequilibres}~:
\begin{equation}\label{diagssvd}D=\diag(d_1,d_2,d_3)\in\mathscr{D}\,\,\,\text{if and only if}\,\,\,d_1\geq d_2\geq |d_3|.\end{equation}

\section{Equilibria of the BGK operator}\label{paragraphequilibria}

In this section we determine the equilibria for the BGK operator:
\begin{equation}\label{bgkoperator}
Q_{BGK}(f):=\rho M_{J_f}-f,
\end{equation}
that is to say the distributions $f$ such that $Q_{BGK}(f)=0$. In Section \ref{characequilibria} we characterise these equilibria (Theorem \ref{solutionscompatibility}) and show that for them to exist, compatibility equations must be fulfilled. These compatibility equations depend on the density $\rho$. Therefore, for different values of the density $\rho$, there exists different equilibria. These will be determined in Section \ref{paragraphdiagonalsolutions} by studying the compatibility equations. A full description of the equilibria of the BGK operator is finally given in Corollary \ref{phasediagramcorollary}.

\subsection{Characterisation of the equilibria and compatibility equations}\label{characequilibria}
The main result of this section is Theorem \ref{solutionscompatibility} which gives all the equilibria of the BGK operator \eqref{bgkoperator}. Before stating and proving it we will need the following lemma which is the analog of lemma 4.4 in \cite{degondfrouvellemerino17}. The proof of this lemma is an application of the results presented in Section \ref{SO3technique}.

\begin{lemma}[Consistency relations]\label{c1c2} The following holds:\begin{enumerate}[(i)]
\item There exists a function $c_1=c_1(\alpha)$ defined for all $\alpha\in\R$ such that for all $\Lambda\in SO_3(\R)$, 
\begin{equation}\label{defc1}c_1(\alpha)\Lambda=\langle A\rangle_{M_{\alpha\Lambda}}.\end{equation}
The function $c_1$ can be explicitly written $c_1(\alpha)=\frac{1}{3}\big\{ (2\cos\theta+1)\big\}_{\alpha}$ where $\{\cdot\}_\alpha$ denotes the mean with respect to the probability density 
\begin{equation}\label{PDF1}\theta\in[0,\pi]\longmapsto \frac{\sin^2(\theta/2) e^{\alpha\cos\theta}}{\int_0^\pi\sin^2(\theta'/2) e^{\alpha\cos\theta'}\,d\theta'}.\end{equation}
\item Consider the set $\mathscr{B}\subset \MM_3(\R)$ defined by:
\[\mathscr{B}:=\left\{B=P\left(\begin{array}{ccc}1 &  &  \\ & 0 &  \\ &  & 0\end{array}\right)Q,\,\,\,\,\,P,Q\in SO_3(\R)\right\}=\{p\otimes q,\,\,\,p,q\in\Sph^2\}.\]
There exists a function $c_2=c_2(\alpha)$ defined for all $\alpha\in\R$ such that for all $B\in\mathscr{B}$, 
\begin{equation}\label{defc2}c_2(\alpha)B=\langle A\rangle_{M_{\alpha B}}.\end{equation}
The function $c_2$ can be explicitly written: $c_2(\alpha)=[ \cos\phi]_\alpha$, where $[\cdot]_\alpha$ denotes the mean with respect to the probability density 
\begin{equation}\label{PDF2}\varphi\in[0,\pi]\,\longmapsto \frac{\sin\varphi\,e^{\frac{\alpha}{2}\cos\varphi}}{\int_0^\pi \sin\varphi'\,e^{\frac{\alpha}{2}\cos\varphi'}\,d\varphi'}.\end{equation}
\end{enumerate}
\end{lemma}

\begin{rem} The relevance of the set $\mathscr{B}$ will become apparent in Proposition \ref{necessarycondition}.
\end{rem}

\begin{proof}
\begin{enumerate}[(i)]
\item Using the left invariance of the Haar measure, it is enough to prove the result for $\Lambda= I_3$, since 
\[\langle A\rangle_{M_{\alpha\Lambda}}=\frac{\int_{SO_3(\R)} A e^{\alpha A\cdot \Lambda}\,dA}{\int_{SO_3(\R)} e^{\alpha\Lambda\cdot A}\,dA}=\Lambda\frac{\int_{SO_3(\R)} \Lambda^T A e^{\alpha \Lambda^TA\cdot I_3}\,dA}{\int_{SO_3(\R)} e^{\alpha \Lambda^TA\cdot I_3}\,dA}=\Lambda\langle A\rangle_{M_{\alpha I_3}}.\]
When $\Lambda=I_3$, Lemma \ref{diagonal} first ensures that $\langle A\rangle_{M_{\alpha I_3}}$ is diagonal, then the change of variable $A'=P^{12}A(P^{12})^T$ (see Definition \ref{changeofvariable}) shows that:
\[\langle a_{11}\rangle_{M_{\alpha I_3}}=\langle a_{22}\rangle_{M_{\alpha I_3}}.\]
Proceeding analogously with the other coefficients we have that $\langle A\rangle_{M_{\alpha I_3}}$ is proportional to $I_3$, i.e. there exists $c_1=c_1(\alpha)\in\R$ such that
\begin{equation}\label{c1eqnannexe}
c_1(\alpha)I_3=\langle A\rangle_{\alpha I_3}.
\end{equation}
The parametrisation of $SO_3(\R)$ using Rodrigues' formula \eqref{rodrigue} then gives the explicit expression of $c_1$ by taking the trace in Equation \eqref{c1eqnannexe} and using that for $A=A(\theta,\mathbf{n})$, $\Tr(A)=2\cos\theta+1$.
\item As before, using the left and right invariance of the Haar measure it is enough to prove the result for $B=\diag(1,0,0)$. Now if $D=\diag(a,b,-b)$ for $a,b\in\R$, then the change of variable $A\mapsto P^{23}A(P^{23})^T$ followed by the change of variable $A\mapsto D^{23}A$ (see Definition \ref{changeofvariable}) show that 
\[\int_{SO_3(\R)} a_{22} e^{D\cdot A}\,dA=-\int_{SO_3(\R)} a_{33} e^{D\cdot A}\,dA,\]
which proves with lemma \ref{diagonal} that $\langle A\rangle_{M_D}$ is diagonal of the form $\diag(\tilde{a},\tilde{b},-\tilde{b})$ for $\tilde{a},\tilde{b}\in\R$. Similarly, if $D=\diag(a,b,b)$ then $\langle A\rangle_{M_D}$ is of the form $\diag(\tilde{a},\tilde{b},\tilde{b})$. These two results prove that $\langle A\rangle_{M_{\alpha B}}$ is proportional to $B$, i.e. there exists $c_2=c_2(\alpha)\in\R$ such that \eqref{defc2} holds. 
The parametrisation of $SO_3(\R)$ coming from the isomorphism \eqref{SOSphere} then gives the explicit expression of $c_2$ by taking $B=\diag(1,0,0)$ in Equation \eqref{defc2}. First, using the change of variable $A\mapsto P^{13}A(P^{13})^T$ it holds that, 
\[c_2(\alpha)=\frac{1}{Z}\int_{SO_3(\R)}a_{11}e^{\frac{\alpha}{2}a_{11}}\,dA=\frac{1}{Z}\int_{SO_3(\R)}a_{33}e^{\frac{\alpha}{2}a_{33}}\,dA\]
where
\[Z=\int_{SO_3(\R)}e^{\frac{\alpha}{2}a_{11}}\,dA=\int_{SO_3(\R)}e^{\frac{\alpha}{2}a_{33}}\,dA.\]
Then, using the parametrisation \eqref{parametrisationSOSphere}, it follows that:
\[c_2(\alpha)=\frac{\int_{0}^\pi \cos\varphi \sin\varphi e^{\frac{\alpha}{2}\cos\varphi}\,d\varphi}{\int_{0}^\pi \sin\varphi e^{\frac{\alpha}{2}\cos\varphi}\,d\varphi}.\]
\end{enumerate}
\end{proof}

\begin{rem}
We could alternatively use one of the two parametrisations of $SO_3(\R)$ given in Section \ref{volumeforms} or the quaternion formulation to prove that $\langle A\rangle_{\alpha I_3}$ and $\langle A\rangle_{\alpha B}$ are proportional to $I_3$ and $B$. However, the proof that we have just presented here holds in any dimension (the value of the constants $c_1(\alpha)$ and $c_2(\alpha)$ depends on the dimension but not the form of the matrices) whereas the volume forms and the quaternion formulation strongly depend on the dimension $n=3$.  
\end{rem}

We can now state the main result of this section:
\begin{theorem}[Equilibria for the homogeneous Body-Orientation BGK equation]\label{solutionscompatibility}
Let $\rho~\in~\R_+$ be a given density. The equilibria of the spatially homogeneous BGK equation \eqref{bgkh} are the distributions of the form $f=\rho M_J$ where $J\in\MM_3(\R)$ is a solution of the matrix compatibility equation:
\begin{equation}\label{compatibility}
J=\rho\langle A\rangle_{M_J}.
\end{equation}
The solutions of the compatibility equation \eqref{compatibility} are:
\begin{enumerate}
\item the matrix $J=0$,
\item the matrices of the form $J=\alpha \Lambda$ with $\Lambda\in SO_3(\R)$ and where $\alpha\in\R$ satisfies the scalar compatibility equation 
\begin{equation}\label{compatibilityc1}\alpha=\rho c_1(\alpha),\end{equation}
\item the matrices of the form $J=\alpha B$ where $B\in \mathscr{B}$ and where $\alpha\in\R$ satisfies the scalar compatibility equation 
\begin{equation}\label{compatibilityc2}\alpha=\rho c_2(\alpha),\end{equation}
\end{enumerate}
where the set $\mathscr{B}$ and the functions $c_1$ and $c_2$ are defined in Lemma \ref{c1c2}.
\end{theorem}

\begin{rem}
Notice that the existence of a non-zero solution for the scalar compatibility equations \eqref{compatibilityc1} and \eqref{compatibilityc2} is not guaranteed for all values of $\rho>0$ . The existence of non-zero solutions for these equations will be explored in Section \ref{paragraphdiagonalsolutions}. They will determine the existence of equilibria for Equation \eqref{bgkh} for a given value of $\rho$ (Corollary \ref{phasediagramcorollary}).
\end{rem}

\begin{rem} The fact that these matrices are solutions of the matrix compatibility equation \eqref{compatibility} follows directly from the consistency relations \eqref{defc1} and \eqref{defc2} as it will be shown in the proof of Theorem \ref{solutionscompatibility}. The main difficulty of the proof is therefore the necessary condition: we will prove that a solution of the matrix compatibility equation \eqref{compatibility} is necessarily of one of the forms listed in Theorem \ref{solutionscompatibility}.\end{rem}

The proof of this theorem will use the two following propositions. The first one and its corollary (Proposition \ref{equivalence} and Corollary \ref{equivalenceSSVD}) show that the compatibility equation \eqref{compatibility} can be reduced to a compatibility equation on diagonal matrices (equation \eqref{compatibilityD}). The second one (Proposition \ref{necessarycondition}) provides a necessary condition for a diagonal matrix to be a solution of \eqref{compatibilityD}. The proof of Proposition \ref{necessarycondition} is deferred to the next section.
\begin{proposition}[Orbital reduction] \label{equivalence}
The following equivalence holds: $J\in\MM_3(\R)$ is a solution of the matrix compatibility equation \eqref{compatibility} if and only if for all $J'\in\Orb(J)$, $J'$ is a solution of the matrix compatibility equation \eqref{compatibility}.
\end{proposition}
\begin{proof}
This is a consequence of the left and right invariance of the Haar measure which ensures that for any $J\in\MM_3(\R)$ and any $P,Q\in SO_3(\R)$~:
\[\langle PAQ\rangle_{M_J}=\langle A\rangle_{M_{PJQ}.}\]\end{proof}
Since the diagonal part of the SSVD of a matrix $J$ is in the orbit of $J$, we obtain the following corollary:
\begin{corollary}[Reduction to diagonal matrices]\label{equivalenceSSVD}
Let $J\in\MM_3(\R)$ with SSVD given by $J~=~PDQ$. The following equivalence holds: $J$ is a solution of \eqref{compatibility} if and only if $D$ is a solution of \eqref{compatibility}.
\end{corollary}
We will therefore consider only the following problem in dimension 3: find all the diagonal matrices $D\in\MM_3(\R)$ such that
\begin{equation}\label{compatibilityD}
\left\{\begin{array}{l}
D=\rho\langle A\rangle_{M_D}\\
D\in\mathscr{D},\end{array}\right.
\end{equation}
where the set $\mathscr{D}$ is the subset of diagonal matrices which are the diagonal part of a SSVD and is defined by \eqref{diagssvd}. Notice that Equation \eqref{compatibilityD} is just Equation \eqref{compatibility} restricted to the set~$\mathscr{D}$.
\begin{rem}\label{diagonalsolutionsandtheothers} The diagonal part $D\in\MM_3(\R)$ of a SSVD of a matrix $J\in\MM_3(\R)$ is unique so the problems \eqref{compatibility} and \eqref{compatibilityD} are equivalent. Notice that there might be other diagonal matrices in $\Orb(J)$ (take for example $J$ diagonal which does not satisfy the conditions \eqref{diagssvd}). However the diagonal part of any SSVD of these matrices is $D$~: the diagonal part of the SSVD characterises the orbit of a matrix. In the following, we will find all the diagonal solutions of \eqref{compatibility} (i.e. the solutions of \eqref{compatibilityD} without the restriction $D\in\mathscr{D}$) and then only consider the ones which belong to $\mathscr{D}$. For instance we will see that there are solutions of \eqref{compatibility} of the form $\diag(0,-\alpha,0)$ where $\alpha>0$. The  diagonal part of their SSVD is $\diag(\alpha,0,0)$ and is a solution of \eqref{compatibilityD}. 
\end{rem}

\begin{rem} A diagonal solution $D$ of the matrix compatibility equation \eqref{compatibility} verifies that $D/\rho$ belongs to the set: 
\[\Omega=\Big\{D=\diag(d_1,d_2,d_3),\,\,\,\exists\,\,f\in\mathcal{P}(SO_3(\R)),\,\,\,J_f=D\Big\}\subset\mathscr{D}_3(\R),\]
where $\mathcal{P}(SO_3(\R))$ is the set of probability measures on $SO_3(\R)$. The set $\diag^{-1}(\Omega)\subset\R^3$ is exactly the tetrahedron $\mathscr{T}$ defined as the convex hull of the points $(\pm1,\pm1,\pm1)$ with an even number of minuses (which we will call \textit{Horn's tetrahedron}). It is a consequence of Horn's theorem \cite[Theorem 8]{horn54} which states that $\mathscr{T}$ is exactly the set of vectors which are the diagonal of an element of $SO_3(\R)$. It ensures that if $f$ is a probability measure, we have by convexity of $\mathscr{T}$~: 
\[\int_{SO_3(\R)} f(A) A\,dA\in\diag(\mathscr{T})\]
and therefore $\diag^{-1}(\Omega)\subset\mathscr{T}$. Conversely, taking the Dirac deltas $\delta_{I_3}$ and similarly for the other vertices of $\mathscr{T}$, we see that the four vertices of Horn's tetrahedron belong to $\diag^{-1}(\Omega)$. Since $\Omega$ is convex, we conclude that $\mathscr{T}\subset\diag^{-1}(\Omega)$.
\end{rem} 

The diagonal solutions of the matrix compatibility equation \eqref{compatibility} satisfy the following necessary condition.

\begin{proposition}\label{necessarycondition}
The diagonal solutions of the compatibility equation \eqref{compatibility} are necessarily of one of the following the types : 
\begin{enumerate}[(a)]
\item $D=0$.
\item $D=\alpha\diag(\pm1,\pm1,\pm1)$ with an even number of minus signs and where $\alpha\in\R\setminus\{0\}$.\

If $\alpha\in(0,+\infty)$, the diagonal part of the SSVD of these diagonal matrices is equal to $D=\alpha I_3$.\\
If $\alpha\in(-\infty,0)$, the diagonal part of the SSVD of these diagonal matrices is equal to $D=\alpha\diag(-1,-1,1)=|\alpha|\diag(1,1,-1)$.
\item $D=\alpha\diag(\pm1,0,0)$ and the matrices obtained by permutation of the diagonal coefficients and where $\alpha\in\R\setminus\{0\}$.\

The diagonal part of the SSVD of these diagonal matrices is equal to $D=\diag(|\alpha|,0,0)$. 

\end{enumerate}
\end{proposition}
Section \ref{paragraphdiagonalsolutions} will be devoted to the proof of this proposition.
We are now ready to prove Theorem \ref{solutionscompatibility}.

\begin{proof}[\sc Proof (of Theorem \ref{solutionscompatibility})] An equilibria of the BGK equation is of the form 
\[f=\rho M_J,\]
where 
\[J=J_f =\rho\langle A\rangle_{M_J}.\]
It is straightforward to check that $J=0$ is a solution of \eqref{compatibility}. Now, let $D$ a matrix of one the form described in Proposition \ref{necessarycondition} with a parameter $\alpha\in\R$. For instance, for a matrix of type \textit{(c)} like $D=\alpha\diag(0,-1,0)$, thanks to Lemma \ref{c1c2} we have:
\[D=\rho\langle A\rangle_{M_D}\,\Longleftrightarrow\,D=\rho c_2(\alpha) \diag(0,-1,0)\,\Longleftrightarrow\,\alpha=\rho c_2(\alpha).\]
Similarly for the other diagonal matrices of type \textit{(c)}, we prove that they are solution of the matrix compatibility equation \eqref{compatibility} if and only if their parameter $\alpha\in\R$ is solution of the scalar compatibility equation \eqref{compatibilityc2}. Analogously one can check that the diagonal matrices of type \textit{(b)} are solutions of the matrix compatibility equation \eqref{compatibility} if and only if their parameters $\alpha\in\R$ are solutions of the scalar compatibility equation \eqref{compatibilityc1}. This yields all the diagonal solutions of \eqref{compatibility}. Now, the solutions of \eqref{compatibility} are exactly the matrices $J\in \Orb(D)$ where $D$ is a diagonal solution of \eqref{compatibility} and the set $\Orb(D)\subset\MM_3(\R)$ is the orbit of $D$ defined in the introduction. We conclude by noticing that if $D$ is of type $(b)$ then $\Orb(D)= SO_3(\R)$ and if $D$ is of type $(c)$ then $\Orb(D)=\mathscr{B}$.
\end{proof}

\begin{rem}\label{sufficientcondition} When applied to diagonal matrices, the last part of Theorem \ref{solutionscompatibility} states that the diagonal solutions of \eqref{compatibility} are necessarily of one of the types (a), (b) or (c) defined in Proposition \ref{necessarycondition} and that, it holds that
\begin{enumerate}
\item the matrix $0$ is always a solution of \eqref{compatibility}, 
\item a matrix of type \textit{(b)} is a solution of \eqref{compatibility} iff its parameter $\alpha\in\R\setminus\{0\}$ satisfies \eqref{compatibilityc1},
\item a matrix of type \textit{(c)} is a solution of \eqref{compatibility} iff its parameter $\alpha\in\R\setminus\{0\}$ satisfies \eqref{compatibilityc2}.
\end{enumerate}
\end{rem}

\subsection{Proof of Proposition \ref{necessarycondition}}\label{paragraphdiagonalsolutions}

The proof of Proposition \ref{necessarycondition} is based on two results. The first one
has been proved in \cite[Section 4]{wanghoffman08} to study the nematic alignment of polymers in higher dimensional spaces:
\begin{theorem}[\cite{wanghoffman08}]\label{wanghoffman}
Let $n\geq3$, $b\in\R_+$ and $\mathbf{s}=(s_1,s_2,\ldots,s_n)\in \R^n$ a solution of the nonlinear system 
\begin{equation}\label{compatibilitywanghoffman}s_j=\langle m_j^2\rangle_{g_{\mathbf{s},b}},\,\,\,\,\,j=1,\ldots,n,\end{equation}
where the average is taken with respect to the PDF on the sphere $\Sph^{n-1}$~:
\begin{equation}\label{gsb}g_{\mathbf{s},b}(m_1,\ldots,m_n):=\frac{1}{Z}\exp\left(b\sum_{j=1}^{n} s_j m_j^2\right),\end{equation}
where $Z$ is the normalisation constant which ensures that $g_{\mathbf{s},b}$ is a PDF on the sphere $\Sph^{n-1}$. Then $\Card\{s_1,s_2,\ldots,s_n\}\leq2$.\end{theorem}
The second tool that we will use to prove Proposition \ref{necessarycondition} is an isomorphism between $SO_3(\R)$ and the space of unitary quaternions which transforms the compatibility equation \eqref{compatibilityD} into the compatibility equation \eqref{compatibilitywanghoffman} studied in Theorem \ref{wanghoffman}.
\begin{proposition}\label{SO3quaternions} \
\begin{enumerate}
\item There is an isomorphism between the group $SO_3(\R)$ and the quotient group $\HH/\pm1$, where $\HH$ is the group of unit quaternions. Since $\HH$ is homeomorphic to $\Sph^3$, there is an isomorphism $\Phi$~: 
\[\Phi : \Sph^3/\pm1 \longrightarrow SO_3(\R).\]
Moreover $\Phi$ is an isometry in the sense that it maps the volume form of $\Sph^3/\pm1$ (defined as the image measure of the usual measure on $\Sph^3$ by the projection on the quotient space) to the volume form on $SO_3(\R)$: for all measurable function $f$ on $SO_3(\R)$,
\[\int_{\Sph^3/\pm1} f\big(\Phi(q)\big)\,dq = \int_{SO_3(\R)} f(A)\,dA.\]
\item There is a linear isomorphism between the vector space $\MM_3(\R)$ and the vector space $\SS_4^0(\R)$ of trace free symmetric matrices of dimension 4:
\[\phi~: \MM_3(\R) \longrightarrow \SS_4^0(\R),\]
such that 
for all $J\in\MM_3(\R)$, and $q\in \HH/\pm1$, 
\[\frac{1}{2} J\cdot \Phi(q) = q\cdot \phi(J)q.\]
The first dot product is defined by Equation \eqref{produitscalaire} and the second one is the usual dot product in $\R^4$. 
 \item For all $q\in\HH/\pm1$, it holds that $\phi\big(\Phi(q)\big)=q\otimes q-\frac{1}{4}I_4$.
 \item The isomorphism $\phi$ preserves the diagonal structure: if $D=\diag(d_1,d_2,d_3)$ then, 
\[\phi(D)=\frac{1}{4}\left(\begin{array}{cccc}d_1+d_2+d_3 & 0 & 0 & 0 \\ 0 & d_1-d_2-d_3 & 0 & 0 \\ 0 & 0 & -d_1+d_2-d_3 &  \\ & 0 & 0 & -d_1-d_2+d_3\end{array}\right)\]
and if $Q=\diag(s_1,s_2,s_3,s_4)$ with $s_1+s_2+s_3+s_4=0$, then 
\[\phi^{-1}(Q)=2\left(\begin{array}{ccc}
{s_1+s_2} & 0 & 0\\ 
 0 & {s_1+s_3} & 0 \\
 0 & 0 & {s_1+s_4}
 \end{array}
 \right).\]

\end{enumerate}
\end{proposition}
The proof of this proposition can be found in appendix \ref{quaternionsandrotations}. We are now ready to prove Proposition \ref{necessarycondition}.
\begin{proof}[\sc Proof (of Proposition \ref{necessarycondition})]
Using the first and second points of Proposition \ref{SO3quaternions}, it holds that
\[\int_{SO_3(\R)} A e^{A\cdot D}\,dA = \int_{\Sph^3/\pm1} \Phi(q)e^{\Phi(q)\cdot D}\,dq = \int_{\Sph^{3}/\pm1}\Phi(q)e^{2q\cdot\phi(D)q}\,dq.\]
The compatibility equation \eqref{compatibilityD} then becomes:
\[D=\frac{\rho}{Z}\int_{\Sph^{3}/\pm1}\Phi(q)e^{2q\cdot\phi(D)q}\,dq.\]
 Applying the isomorphism $\phi$ defined in Proposition \ref{SO3quaternions} to this last equation, we obtain thanks to the third point of Proposition \ref{SO3quaternions}~:
 \[\phi(D)=\frac{\rho}{Z}\int_{\Sph^{3}/\pm1}\phi\big(\Phi(q)\big)e^{2q\cdot\phi(D)q}\,dq=\frac{\rho}{Z}\int_{\Sph^{3}/\pm1}\left(q\otimes q-\frac{1}{4}I_4\right)e^{2q\cdot\phi(D)q}\,dq.\]
Using the fourth point of Proposition \ref{SO3quaternions}, we then obtain the following equivalent problem: find all the trace-free diagonal matrices $Q=\diag(s_1,s_2,s_3,s_4)$ of dimension 4 such that 
\[Q=\rho\frac{\int_{\Sph^3/\pm1} e^{\sum_{i=1}^{4} 2s_iq_i^2}(q\otimes q-\frac{1}{4} I_4)\,dq}{Z},\]
where $Z$ is a normalisation constant:
\[Z:=\int_{\Sph^3/\pm1} e^{\sum_{i=1}^{4} 2s_iq_i^2}\,dq.\]
Equivalently, defining for $i\in\{1,2,3,4\}$~:
\[s_i':=\frac{s_i}{\rho}+\frac{1}{4},\]
we want to solve the system of compatibility equations:~
\begin{equation}\label{compatibilitysi}s_i' = \int_{\Sph^3/\pm1} q_i^2g_{\mathbf{s}',2\rho}(q)\,dq,\,\,\,\,\,\,\,i=1,2,3,4,\end{equation}
where $\mathbf{s}'=(s_1',s_2',s_3',s_4')$ and $g_{\mathbf{s}',2\rho}$ is given by \eqref{gsb}. Thanks to Theorem \ref{wanghoffman}, we conclude that if $\mathbf{s}'$ is a solution of \eqref{compatibilitysi}, then the coefficients $s_1',s_2',s_3',s_4'$ can take at most two distinct values. So, the same result holds for the coefficients $s_1,s_2,s_3,s_4$. Now thanks to the fourth point of Proposition \ref{SO3quaternions}, we only have the following possibilities:
\begin{enumerate}[$\bullet$]
\item if $s_1=s_2=s_3=s_4=0$, then 
\[D=\phi^{-1}(Q)=0,\]
\item if $s_1=3\alpha/4$ and $s_2=s_3=s_4=-\alpha/4$ for $\alpha\in\R$, then 
\[D=\phi^{-1}(Q)=\alpha I_3,\]
\item if $s_2=3\alpha/4$ and $s_1=s_3=s_4=-\alpha/4$ for $\alpha\in\R$, then 
\[D=\phi^{-1}(Q)=\alpha\left(\begin{array}{ccc}1 & 0 & 0 \\0 & -1 & 0 \\0 & 0 & -1\end{array}\right),\]
and similarly by permuting the diagonal elements when $s_3=3\alpha/4$ and when the other elements are equal $s_1=s_2=s_4=-\alpha/4$ or when $s_4=3\alpha/4$ and $s_1=s_2=s_3=-\alpha/4$,
\item if $s_1=s_2=\alpha/4$ and $s_3=s_4=-\alpha/4$ for $\alpha\in\R$, then 
\[D=\phi^{-1}(Q)=\alpha\left(\begin{array}{ccc}1 & 0 & 0 \\0 & 0 & 0 \\0 & 0 & 0\end{array}\right),\] 
and similarly by permuting the diagonal elements when $s_1=s_3=\alpha/4$ and when $s_2=s_4=-\alpha/4$ or $s_1=s_4=\alpha/4$ and $s_2=s_3=-\alpha/4$.
\end{enumerate}
The computation of the SSVD for these matrices is an easy computation.
This concludes the proof of Proposition \ref{necessarycondition}.\end{proof}

\subsection{Determination of the equilibria for each density $\rho$}

In Theorem \ref{solutionscompatibility} we saw that the BGK operator can have three types of equilibria. The uniform equilibria $f=\rho$ (corresponding to $J=0$) is always an equilibrium. However, the existence of the other two types of equilibria depends on Equations \eqref{compatibilityc1} and \eqref{compatibilityc2} having a solution for a given $\rho$. Therefore the existence of these types of equilibria will depend on the value of $\rho$. In this section we will determine the existing equilibria for each value of $\rho$. In particular, we will draw the phase diagram for $\rho$ and $\alpha$, that is to say the parametrised curves defined by Equations \eqref{compatibilityc1} and \eqref{compatibilityc2} in the plane $(\rho,\alpha)$ (see Figure \ref{phasediagram}). We first prove the following proposition.

\begin{proposition}\label{alphasurc} Let $\rho_c:=6$. 
\begin{enumerate}[(i)]
\item The function $\alpha\mapsto \alpha/c_1(\alpha)$ is well-defined on $\R$, its value at zero is $\rho_c$. Moreover, there exists $\alpha^*>0$ such that this function is decreasing on $(-\infty,\alpha^*]$ and increasing on $[\alpha^*,+\infty)$. Defining  $\rho^*:=\alpha^*/c_1(\alpha^*)$, it holds that $\rho^*<\rho_c$.
\item The function $\alpha\mapsto \alpha/c_2(\alpha)$ is even. It is decreasing on $(-\infty,0)$, increasing on $(0,\infty)$ and its value at zero is $\rho_c$.
\item We have the following asymptotic behaviours:
\[\frac{\alpha}{c_1(\alpha)}\underset{\alpha\to+\infty}{\sim}\alpha+1,\]
\[\frac{\alpha}{c_2(\alpha)}\underset{\alpha\to+\infty}{\sim}\alpha+2.\]

\end{enumerate}
\end{proposition}

\begin{proof} The idea of the proof is taken from \cite{wanghoffman08}.
\begin{enumerate}[(i)]
\item Since
\[\frac{d}{d\theta}\Big\{\sin^2(\theta/2)\sin\theta\Big\}=\sin^2(\theta/2)(1+2\cos\theta),\]
an integration by parts shows that: 
\[\frac{\alpha}{c_1(\alpha)}=3\frac{\int_0^\pi \sin^2(\theta/2)e^{\alpha\cos\theta}\,d\theta}{\int_0^\pi \sin^2(\theta/2)\sin^2\theta e^{\alpha\cos\theta}\,d\theta}=\frac{3}{\{ \sin^2\theta\}_\alpha}.\]
It proves that $\alpha$ and $c_1(\alpha)$ have the same sign for all $\alpha\in\R$. Then we define function $m~:\alpha\mapsto \{ \sin^2\theta\}_\alpha=3 c_1(\alpha)/\alpha$ which satisfies the property: 
\[m'(\alpha)=0\,\,\Longrightarrow m''(\alpha)<0,\]
since
\[m'(\alpha)=\{ \sin^2\theta\cos\theta\}_\alpha-\{\sin^2\theta\}_\alpha\{ \cos\theta\}_\alpha,\]
and
\[m''(\alpha)=-\Var_\alpha(\cos^2\theta)-2\{\cos\theta\}_\alpha m'(\alpha),\]
where $\Var_\alpha$ is the variance for the probability density \eqref{PDF1}. This property implies that $\alpha/c_1(\alpha)$ has only one critical point which is a global minimum. This minimum is attained at a point $\alpha^*>0$ as a simple computation shows that $m'(0)>0$ and consequently $\rho^*<\rho_c$. A simple computation gives $m(0)=\frac{1}{2}$ so $\rho_c=6$.
\item We have similarly: 
\begin{equation}\label{PierreN}\frac{\alpha}{c_2(\alpha)} = 4\frac{\int_0^\pi \sin\varphi e^{\frac{\alpha}{2}\cos\varphi}\,d\varphi}{\int_0^\pi \sin^3\varphi e^{\frac{\alpha}{2}\cos\varphi}\,d\varphi}=\frac{4}{[\sin^2\varphi]_\alpha},\end{equation}
from which we can easily see that $\alpha\mapsto \alpha/c_2(\alpha)$ is even and has only one minimum attained at $\alpha=0$. A simple computation shows that its value at 0 is $\rho_c=6$. 
\item The behaviour at infinity is obtained by Laplace's method: with the change of variable $s=1-\cos\theta$ on $[0,\pi]$, we get
\[\frac{\alpha}{c_1(\alpha)}=3\frac{e^{\alpha}\int_0^2 e^{-\alpha s}\frac{s}{2\sqrt{1-(1-s)^2}}\,ds}{e^\alpha\int_0^2 e^{-\alpha s}\frac{s}{2}\sqrt{1-(1-s)^2}\,ds}\underset{\alpha\to+\infty}{\sim}\alpha+1.\]
With the same method we have:
\[\frac{\alpha}{c_2(\alpha)}=4\frac{\int_0^2 e^{-\frac{\alpha}{2}s}\,ds}{\int_0^2 e^{-\frac{\alpha}{2}s}(1-(1-s)^2)\,ds}\underset{\alpha\to+\infty}{\sim}\alpha+2.\]
\end{enumerate}
\end{proof}

Thanks to Proposition \ref{alphasurc} and Theorem \ref{solutionscompatibility} we can now fully describe the equilibria of the BGK operator. A graphical representation of this result is given by the phase diagram depicted in Figure \ref{phasediagram}~:
\begin{corollary}[Equilibria of the BGK operator, depending on the density $\rho$]\label{phasediagramcorollary} The set of equilibria of the BGK operator \eqref{bgkoperator} depends on the value of $\rho$. In particular we need to distinguish three regions $\rho\in(0,\rho^*)$, $\rho\in(\rho^*,\rho_c)$ and $\rho>\rho_c$ where $\rho^*$ and $\rho_c$ are defined in Proposition \ref{alphasurc}. We have the following equilibria in each region:
\begin{enumerate}[$\bullet$]
\item For $0<\rho<\rho^*$, $\alpha=0$ is the unique solution of Equations \eqref{compatibilityc1} and \eqref{compatibilityc2} and therefore the only equilibrium is the uniform equilibrium $f^\mathrm{eq}=\rho$.
\item For $\rho=\rho^*$, in addition to the uniform equilibrium, there is a family of anisotropic equilibria given by $f^\mathrm{eq}=\rho^* M_{\alpha^* \Lambda}$ where $\Lambda\in SO_3(\R)$ and $\alpha^*=\rho^* c_1(\alpha^*)$.
\item For $\rho^*<\rho<\rho_c$, the compatibility equation \eqref{compatibilityc1} has two solutions $\alpha_+$ and $\alpha_-$ with $0<\alpha_-<\alpha_+$ which give, in addition to the uniform equilibrium, two families of anisotropic equilibria : $f^\mathrm{eq}=\rho M_{\alpha_+\Lambda}$ and $f^\mathrm{eq}=\rho M_{\alpha_-\Lambda}$ with $\Lambda\in SO_3(\R)$.
\item For $\rho=\rho_c$, we have $\alpha_-=0$.
\item For $\rho>\rho_c$, Equation \eqref{compatibilityc1} has two solutions $\alpha_3<0<\alpha_1$ which give two families of anisotropic equilibria $f^\mathrm{eq}=\rho M_{\alpha_3\Lambda}$ and $f^\mathrm{eq}=\rho M_{\alpha_1\Lambda}$ with $\Lambda\in SO_3(\R)$. Moreover, Equation \eqref{compatibilityc2} has two solutions $-\alpha_2<0<\alpha_2$ which give another family of equilibria: $f^\mathrm{eq}=\rho M_{\alpha_2 B}$ where $B\in\mathscr{B}$. The uniform equilibrium is always an equilibrium. 
\end{enumerate}
\end{corollary}

When an equilibrium is of the form $f^\mathrm{eq}=\rho M_{\alpha\Lambda}$ with parameters $\alpha>0$ and $\Lambda\in SO_3(\R)$ then these parameters can respectively be interpreted as a concentration parameter and a mean body-orientation. They are analogous to the equilibria found in \cite{degondfrouvelleliu15} in the Vicsek case. However, in $SO_3(\R)$, there exist other equilibria which are not of this form. We will see in Section \ref{convergence} that these latter equilibria are always unstable.

\begin{figure}[H]
\centering
\includegraphics[width=0.85\textwidth]{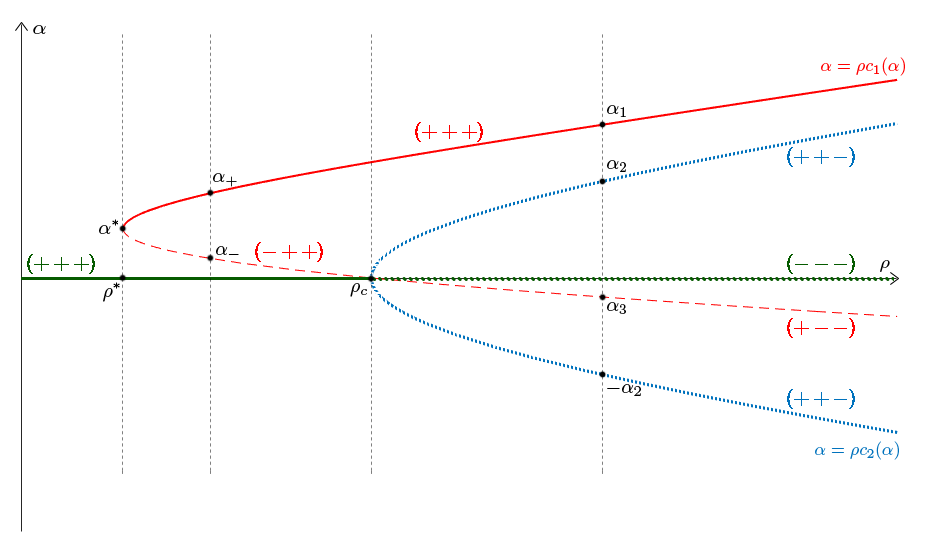}
\caption{Phase diagram for the equilibria of the BGK operator \eqref{bgkoperator}. Depending on the density, there are one, two, three or four branches of equilibria ($\alpha_2$ and $-\alpha_2$ give the same orbit). The uniform equilibrium $f^\mathrm{eq}=\rho$ is always an equilibrium (corresponding to $\alpha=0$, depicted in green). The equilibria of the form $f^\mathrm{eq}=\rho M_{\alpha\Lambda}$, $\Lambda\in SO_3(\R)$ exist for $\rho>\rho^*$ and correspond to the two branches of the red curve $\alpha=\rho c_1(\alpha)$. Finally the equilibria of the form $f^\mathrm{eq}=\rho M_{\alpha B}$, $B\in\mathscr{B}$ exist for $\rho>\rho_c$ and correspond to the two branches of the blue curve $\alpha=\rho c_2(\alpha)$. The dotted and dashed lines correspond to unstable equilibria (as shown in Section \ref{convergence}). The signs are the signature of the Hessian matrix $\Hess V(D)$ defined in Section \ref{convergence} taken at an equilibrium point. The elements $\alpha^*$, $\rho^*$ and $\rho_c$ are defined in Proposition \ref{alphasurc}; the elements $\alpha_+$, $\alpha_-$, $\alpha_1$, $\alpha_2$ and $\alpha_3$ are given in Corollary \ref{phasediagramcorollary}.}
\label{phasediagram}
\end{figure}

Finally the following picture (Figure \ref{dessinequilibres}) is a representation in the space $\R^3$ of the diagonal parts of the SSVDs of the solution of the matrix compatibility equation \eqref{compatibility} when $\rho>\rho_c$. They all belong to the domain $\mathscr{D}$ defined by \eqref{diagssvd} and depicted in orange in Figure \ref{dessinequilibres}.

\begin{figure}[H]
\centering
\includegraphics[width=0.75\textwidth]{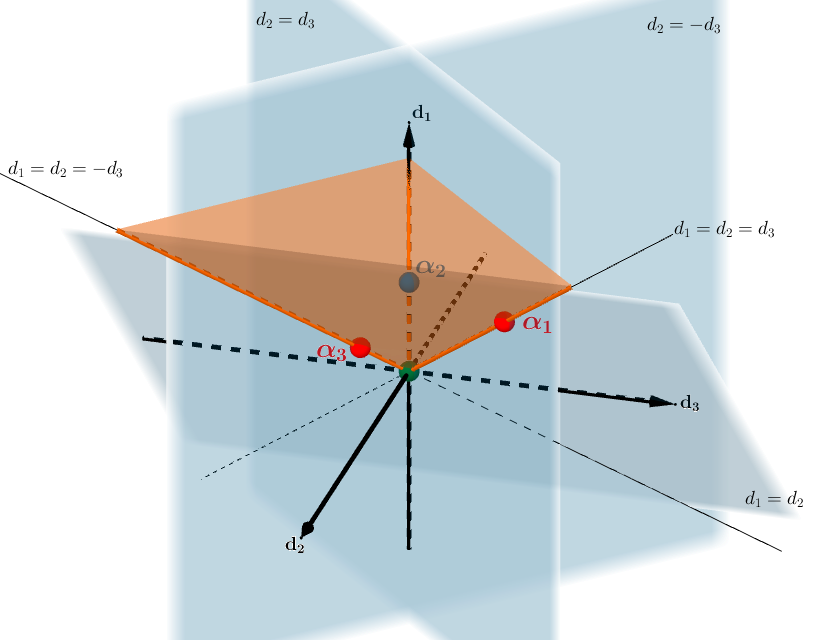}
\caption{The 4 diagonal parts of the SSVDs of the diagonal equilibria seen as elements of the space $\R^3$ for $\rho>\rho_c$, as described in Corollary \ref{phasediagramcorollary}. The ones with non zero determinant are in red (type \textit{(b)} in Proposition \ref{necessarycondition}), the non-zero one with determinant equal to zero is in blue (type \textit{(c)}) and the matrix 0 is in green. They all lie in the domain $\diag^{-1}(\mathscr{D})$ depicted in orange and delimited by the three blue planes $\{d_1=d_2\}$, $\{d_2=d_3\}$ and $\{d_2=-d_3\}$.}
\label{dessinequilibres}
\end{figure}

\section{Convergence to equilibria}\label{convergence}
Now that we know all the equilibria of the spatially homogeneous BGK equation \eqref{bgkh} we proceed to investigate the asymptotic behaviour of $f(t,A)$ as $t\to+\infty$. This problem can be reduced to looking at the asymptotic behaviour of $J_f$ since, if $J_f\to J_\infty\in\MM_3(\R)$, then $f(t)$ will converge as $t\to+\infty$ towards $\rho M_{J_\infty}$ as it can be seen by writing Duhamel's formula for equation \eqref{bgkh}~: 
\begin{equation}\label{duhamel}f(t)=e^{-t}f_0+\rho\int_0^t e^{-(t-s)}M_{J_{f(s)}}\,ds.\end{equation}
The asymptotic behaviour of $J_f$ is much simpler than the one of $f$ since $J_f$ is the solution of the following ODE
\begin{equation}\label{ODEJ}
\frac{d}{dt}J_f=\rho\langle A\rangle_{M_{J_f}}-J_f,\,\,\,\,\,\,J_f(t=0)=J_{f_0},
\end{equation}
as it can be seen by multiplying \eqref{bgkh} by $A\in SO_3(\R)$ and integrating over $SO_3(\R)$. Since $J\in\MM_3(\R)\mapsto M_J\in L^\infty(SO_3(\R))$ is locally Lipschitz, the flow of Equation \eqref{ODEJ} is defined globally in time as a bounded Lipschitz perturbation of the linear system $\frac{d}{dt}J=-J$. \\

Notice that the solutions of the compatibility equation \eqref{compatibility} are exactly the equilibria of the dynamical system \eqref{ODEJ}. We therefore obtain the following proposition:
\begin{proposition}[Equilibria of the BGK operator, equilibria of the ODE \eqref{ODEJ}] A distribution $f^\mathrm{eq}=\rho M_J$ is an equilibrium of the BGK operator \eqref{bgkoperator} if and only if $J\in\MM_3(\R)$ is an equilibrium of the dynamical system \eqref{ODEJ}.\end{proposition}
We will call stable/unstable an equilibrium of the BGK operator \eqref{bgkoperator} such that the associated matrix $J\in\MM_3(\R)$ is a stable/unstable equilibrium of the ODE \eqref{ODEJ}. This section is devoted to the proof of the following theorem:
\begin{theorem}[Convergence towards equilibria]\label{theoremconvergence} Let $\rho\in\R_+$ be such that $\rho\ne\rho^*$ and $\rho~\ne~\rho_c$ (as defined in Proposition \ref{alphasurc}). Let $f_0$ be an initial condition for \eqref{bgkh} and let $J_{f_0}~=~PD_0Q$ be a SSVD. Let $f(t)$ be the solution at time $t\in\R_+$ of the spatially homogeneous BGK equation \eqref{bgkh} with initial condition $f_0$. Let $D(t)$ be the solution at time $t\in\R_+$ of the ODE \eqref{ODEJ} with initial condition $D_0\in\mathscr{D}$. It holds that: 
\[J_{f(t)}=PD(t)Q\]
is a SSVD and there exists a subset $\mathscr{N}_\rho\subset\R^3$ of zero Lebesgue measure such that:
\begin{enumerate}
\item if $D_0\notin\mathcal{N}_\rho$, then $f(t)$ converges as $t\to\infty$ towards an equilibrium $f^\mathrm{eq}$ of the BGK operator \eqref{bgkoperator} of the form $f^\mathrm{eq}=\rho M_{J^\mathrm{eq}}$, where $J^\mathrm{eq}\in\MM_3(\R)$ is of one of the forms described in Theorem \ref{solutionscompatibility}. The convergence is locally exponentially fast in the sense that there exists constants $\delta,K,\mu>0$ such that if $\|J_{f_0}-J^\mathrm{eq}\|\leq\delta$ then for all $t>0$,
\[\forall A\in SO_3(\R),\,\,\,|f(t,A)-f^\mathrm{eq}(A)|\leq e^{-\mu t}\Big(K\rho+|f_0(A)-f^\mathrm{eq}(A)|\Big).\]
\item If $D_0\notin\mathcal{N}_\rho$, we have the following asymptotic behaviours depending on the density $\rho$~:
\begin{enumerate}[(i)]
\item if $0<\rho<\rho^*$, then $D(t)\to0$ as $t\to+\infty$ and, consequently, $f^\mathrm{eq}=\rho$,
\item if $\rho^*<\rho<\rho_c$, then $D(t)\to0$ or $D(t)\to\alpha_+ I_3$ as $t\to+\infty$ and, consequently, $f^\mathrm{eq}=\rho$ or $f^\mathrm{eq}=\rho M_{\alpha_+\Lambda}$ respectively, where $\alpha_+>0$ is defined in corollary \ref{phasediagramcorollary} and $\Lambda:=PQ\in SO_3(\R)$,
\item if $\rho>\rho_c$, then $D(t)\to\alpha_1 I_3$ as $t\to+\infty$ and, consequently, $f^\mathrm{eq}=\rho M_{\alpha_1\Lambda}$ where $\alpha_1>0$ is defined in corollary \ref{phasediagramcorollary} and $\Lambda:=PQ\in SO_3(\R)$.
\end{enumerate}
\end{enumerate}
\end{theorem}
\begin{rem}
The subset $\mathscr{N}_\rho\subset\R^3$ depends on the density $\rho$. This subset will be made explicit in the three cases $\rho<\rho^*$, $\rho^*<\rho<\rho_c$ and $\rho>\rho_c$ in Section \ref{conclusioncompliquee}. If $D_0\in\mathscr{N}_\rho$, then there is convergence towards an unstable equilibrium at a rate which may not be exponential.
\end{rem}

\begin{rem}[Phase transitions] 
Theorem \ref{theoremconvergence} demonstrates a phase transition phenomenon triggered by the density of agents $\rho$~: when $\rho<\rho^*$, the system is disordered (asymptotically in time) in the sense that $J_f\to0$ and we therefore cannot define a mean body-attitude. When the density increases and exceeds the critical value $\rho_c$, the system is self-organised (asymptotically in time and for almost every initial data), in the sense that $J_f\to\alpha\Lambda$ where $\alpha\in\R_+$ and $\Lambda\in SO_3(\R)$  can be respectively interpreted as a concentration parameter and a mean body attitude. When $\rho^*<\rho<\rho_c$ the self-organised and disordered states are both asymptotically stable and the convergence towards one or the other state depends on the initial data. Such ``transition region'' also appears in the Vicsek model, as studied in \cite{degondfrouvelleliu15}, and gives rise to an hysteresis phenomenon. 
\end{rem}

The proof of this theorem can be found at the end of Section \ref{paragraphstability}. It is based on a gradient flow structure for the flux $J_f$ studied in Section \ref{gradientflowstructure}. This structure ensures the convergence of $J_f$ towards a matrix $J^\mathrm{eq}\in\MM_3(\R)$ as $t\to+\infty$ and consequently the convergence of $f(t)$ as $t\to+\infty$ towards an equilibrium. The stability of the equilibria determines which equilibrium can be attained. This question is addressed in Section \ref{paragraphstability}. Additional details about the subset $\mathscr{N}_\rho$ as well as the study of the critical cases $\rho=\rho^*$ and $\rho=\rho_c$ are provided in Section \ref{conclusioncompliquee}.

\subsection{A gradient-flow structure in $\R^3$}\label{gradientflowstructure}

In this section we show that the ODE \eqref{ODEJ} can be reduced to a gradient-flow ODE in $\R^3$. We first show (Section \ref{subsubreduction}) how \eqref{ODEJ} can be reduced to an ODE in $\R^3$, the equilibria of which are linked to the equilibria of \eqref{ODEJ} (and therefore of \eqref{bgkh}). Then we show that this ODE in $\R^3$ has a gradient-flow structure which will allow us to conclude on the asymptotic behaviour of the solution of \eqref{bgkh} (Section \ref{subsubgradient}).

\subsubsection{Reduction to a nonlinear ODE in $\R^3$ and equilibria}\label{subsubreduction}
The ODE \eqref{ODEJ} is a matrix-valued nonlinear ODE (in dimension 9) but, as in the previous section (Proposition \ref{equivalence} and Corollary \ref{equivalenceSSVD}), we will use the left and right invariance of the Haar measure and the SSVD to reduce the problem to a vector-valued nonlinear ODE in dimension 3, as explained in Proposition \ref{reductionODE} and Corollary \ref{ODEssvd}.
\begin{proposition}[Reduction to a nonlinear ODE in dimension 3]\label{reductionODE} Let $J_0\in\MM_3(\R)$ be a given matrix and let $D_0\in\Orb(J_0)$ be diagonal. Let $P,Q\in SO_3(\R)$ such that $J_0=PD_0Q$. Let $J:[0,\infty)\to\MM_3(\R)$ be a $C^1$ curve in $\MM_3(\R)$ with $J(0)=J_0$. For all $t>0$, let $D(t)\in\MM_3(\R)$ be the matrix such that $J(t)=PD(t)Q$. It holds that:
\begin{enumerate}[(i)]
\item $J=J(t)$ is the solution of the ODE \eqref{ODEJ} with initial condition $J(t=0)=J_0$ if and only if $D=D(t)$ is the solution of the same ODE \eqref{ODEJ} with initial condition $D(t=0)=D_0$. 
\item Moreover, if \textit{(i)} holds, then the matrix $D(t)\in\MM_3(\R)$ is diagonal for all time.
\end{enumerate}
\end{proposition}
\begin{proof}
\begin{enumerate}[(i)]
\item Using the left and right invariance of the Haar measure, we see that if the matrix $J(t)\in\MM_3(\R)$ is a solution of \eqref{ODEJ}, then for any $P,Q\in SO_3(\R)$, $PJ(t)Q$ is also a solution (and conversely).
\item Since $D_0$ is diagonal, the fact that $D(t)$ is also diagonal is a consequence of lemma \ref{diagonal} which states that $\langle A\rangle_{M_D}$ is diagonal when $D$ is diagonal.
\end{enumerate}
\end{proof}
For any matrix $J_0$, such a diagonal matrix $D_0$ always exists: we can take the diagonal part of the SSVD of $J_0$. We therefore only have to study the following ODE for diagonal matrices $\mathscr{D}_3(\R)$. Since this vector space is isomorphic to $\R^3$ through the isomorphism $\diag$ defined in the introduction, we obtain two equivalent ODEs: 
\begin{subequations}
\begin{align}
\label{ODED}\frac{d}{dt}D(t)&=\rho\langle A\rangle_{M_{D(t)}}-D(t),\,\,\,\,\,\,D(t=0)=D_0,\\
\label{ODEDR3}\frac{d}{dt}\widehat{D}(t)&=\rho\diag^{-1}\left(\langle A\rangle_{M_{\diag(\widehat{D}(t))}}\right)-\widehat{D}(t),\,\,\,\,\,\,\widehat{D}(t=0)=\widehat{D}_0,
\end{align}
\end{subequations}
where $\widehat{D}_0=\diag^{-1}(D_0)$ and the following equivalence holds : $\widehat{D}(t)=\diag^{-1}(D(t))$ is the solution of \eqref{ODEDR3} if and only if $D(t)=\diag(\widehat{D}(t))$ is the solution of \eqref{ODED}.\\

Note that it is not clear that if $J_0=PD_0 Q$ is a SSVD, then $J(t)=PD(t)Q$ is a SSVD for all $t>0$. The following proposition and corollary ensure that the SSVD is preserved by the dynamical system which will allow us to restrict the domain on which the ODE \eqref{ODED} is posed. 
\begin{proposition}[Invariant manifolds]\label{invariantplanes}
The following subsets of $\R^3$ are invariant manifolds of the dynamical system \eqref{ODEDR3}~:
\begin{enumerate}[$\bullet$]
\item the planes $\big\{(d_1,d_2,d_3)\in\R^3,\,\,d_i+d_j=0\big\}$ for $i\ne j\in\{1,2,3\}$,
\item the planes $\big\{(d_1,d_2,d_3)\in\R^3,\,\,d_i-d_j=0\big\}$ for $i\ne j\in\{1,2,3\}$,
\item the intersections of two of these planes and in particular the lines 
\[\R\left(\begin{array}{c}1\\1\\1\end{array}\right),\,\,\,\,\,\R\left(\begin{array}{c}1\\0\\0\end{array}\right),\,\,\,\,\,\R\left(\begin{array}{c}1\\1\\-1\end{array}\right)\ldots\]
\end{enumerate}
\end{proposition}
\begin{proof}
For $i=2$ and $j=3$, the result has already been proved in the second point of Lemma \ref{c1c2}. The other cases are similar.\end{proof}

\begin{corollary}\label{ODEssvd} Let $J_0\in\MM_3(\R)$ and $J_0=PD_0Q$ be a SSVD with $P,Q\in SO_3(\R)$ and $D_0\in\mathscr{D}_3(\R)$ diagonal. Let $J(t)$ be the solution of the ODE \eqref{ODEJ} with initial condition $J(t=0)=J_0$. Let $D(t)$ the solution of the ODE \eqref{ODED} with initial condition $D(t=0)=D_0$. Then the decomposition $J(t)=PD(t)Q$ is a SSVD for $J(t)$.
\end{corollary}

\begin{proof} The fact that $J(t)=PD(t)Q$ is a consequence of the first point of Proposition \ref{reductionODE}. The fact that it is a SSVD is a consequence of Proposition \ref{invariantplanes} which ensures that the conditions \eqref{diagssvd} remain true for all $t>0$~: $\mathscr{D}$ is stable in the sense that if $D_0\in\mathscr{D}$, then $D(t)\in\mathscr{D}$ for all time $t>0$. This follows from the fact that the image by the isomorphism $\diag$ of the invariant manifolds of \eqref{ODEDR3} described in Proposition \ref{invariantplanes} are invariant manifolds of the dynamical system \eqref{ODED}. These manifolds in $\mathscr{D}_3(\R)$ form the boundary of the subset $\mathscr{D}$.\end{proof}

In conclusion, the study of the asymptotic behaviour of $f(t)$ as $t\to+\infty$  can be reduced to the study of the asymptotic behaviour of the solutions of the ODE \eqref{ODED} posed on the domain $\mathscr{D}$ (see Figure \ref{dessinequilibres}).\\

The following Proposition describes the equilibria of the dynamical system \eqref{ODEG} and is a consequence of the results of Section \ref{paragraphequilibria}.   
\begin{proposition}[Equilibria of the dynamical system \eqref{ODED}]\label{equilibriaODED} The equilibria of the dynamical system \eqref{ODED} are, depending on the density $\rho$~:
\begin{enumerate}[$\bullet$]
\item the matrix $D=0$ for any density $\rho$, 
\item the diagonal matrices of type (b) with parameters $\alpha_+$ and $\alpha_-$ when $\rho^*<\rho<\rho_c$,
\item the diagonal matrices of type (b) with parameters $\alpha_1$ and $\alpha_3$ and the diagonal matrices of type (c) with parameters $\pm\alpha_2$ when $\rho>\rho_c$,
\end{enumerate}
where the types (b) and (c) are defined in Proposition \ref{necessarycondition} and the elements $\alpha_+$, $\alpha_-$, $\alpha_1$, $\alpha_2$ and $\alpha_3$ are defined in Corollary \ref{phasediagramcorollary}.
\end{proposition}
\begin{proof} The equilibria of \eqref{ODED} are the diagonal matrices $D$ such that 
\[\rho\langle A\rangle_{M_D}-D=0,\]
which is exactly Equation \eqref{compatibility} for the diagonal matrices. This equation has been solved in Theorem \ref{solutionscompatibility}, Remark \ref{sufficientcondition} and Corollary \ref{phasediagramcorollary}. 
\end{proof}
\begin{rem} As in the previous section, we have found all the diagonal equilibria of \eqref{ODED}. However, thanks to Corollary \ref{ODEssvd}, only the ones which belong to $\mathscr{D}$ are needed.\end{rem}
The following proposition is a straightforward consequence of the previous results.
\begin{proposition}[Equilibria of the BGK operator, equilibria of the ODE \eqref{ODED}]\label{equilibriaBGKODE} Let $J\in~\MM_3(\R)$ with a SSVD given by $J=PDQ$. The following assertions are equivalent.
\begin{enumerate}[(1)]
\item The distribution $f^\mathrm{eq}=\rho M_J$ is an equilibrium of the BGK operator \eqref{bgkoperator}.
\item The matrix $D$ is an equilibrium of the dynamical system \eqref{ODED} on the domain $\mathscr{D}$.
\item The matrix $D$ is of one of the forms:
\begin{enumerate}[$\bullet$]
\item $D=0$, when $\rho<\rho^*$,
\item $D=0$ or $D=\alpha_- I_3$ or $D=\alpha_+ I_3$, when $\rho^*\leq\rho\leq\rho_c$,
\item $D=0$ or $D=\alpha_1 I_3$ or $D=\alpha_3 \diag(-1,-1,1)$ or $D=\alpha_2\diag(1,0,0)$, when $\rho>\rho_c$,
\end{enumerate}
where $\rho^*$, $\rho_c$, $\alpha_-$, $\alpha_+$, $\alpha_1$, $\alpha_2$ and $\alpha_3$ are defined in Proposition \ref{alphasurc} and Corollary \ref{phasediagramcorollary}.
\end{enumerate}
\end{proposition}

\subsubsection{A gradient-flow structure}\label{subsubgradient}

\begin{lemma}[Gradient-flow structure]\label{gradientflowlemma}
We define the \textit{partition function} of a matrix $J~\in~\MM_3(\R)$~:
\begin{equation}\label{partitionfunction}\mathcal{Z}(J)~:=\int_{SO_3(\R)} e^{J\cdot A}\,dA,\end{equation}
and the potentials $V(J)$ and $\widehat{V}(\widehat{D})$ respectively on $\mathscr{M}_3(\R)$ and on $\R^3$~:
\begin{subequations}
\begin{align}
\label{potential}V(J)&:=\frac{1}{2}\|J\|^2-\rho\log\mathcal{Z}(J),\\
\label{potentialR3}\widehat{V}(\widehat{D})&:=\frac{1}{2}|\widehat{D}|^2-2\rho\log\mathcal{Z}(\diag(\widehat{D}))
\end{align}
\end{subequations}
where $\|\cdot\|$ and $|\cdot |$ are the Euclidean norms respectively on $\MM_3(\R)$ and on $\R^3$. Then we can rewrite equations \eqref{ODED} and \eqref{ODEDR3} into a gradient flow structure as follows: 
\begin{subequations}
\begin{align}
\label{ODEG}\frac{d}{dt}D(t)&=\rho\langle A\rangle_{M_{D(t)}}-D(t)=-\nabla V(D),\\
\label{ODEGR3}\frac{d}{dt}\widehat{D}(t)&=\rho\diag^{-1}\left(\langle A\rangle_{M_{\diag(\widehat{D}(t))}}\right)-\widehat{D}(t)=-\nabla \widehat{V}(\widehat{D}),
\end{align}
\end{subequations}
where $\nabla$ is the gradient operator in $\MM_3(\R)$ endowed with the Riemaniann structure \eqref{produitscalaire} or the gradient operator in $\R^3$ endowed with the usual Euclidean structure.
\end{lemma}
\begin{proof}
The partition function satisfies that for all $J\in\MM_3(\R)$,
\[\nabla (\log\mathcal{Z})(J)=\langle A\rangle_{M_J},\]
since $\nabla(e^{J\cdot A})=Ae^{J\cdot A}$. The result follows in $\MM_3(\R)$. The result in $\R^3$ follows from the fact that for any $w_1,w_2\in\R^3$, it holds that:
\[w_1\cdot w_2=2\diag(w_1)\cdot\diag(w_2)\]
where $\cdot$ denotes the dot product on $\R^3$ and on $\MM_3(\R)$ as defined in \eqref{produitscalaire}.
\end{proof}

\begin{rem}
This gradient-flow structure on $J_f$ is specific to the BGK equation and does not hold for the Fokker-Planck operator (as shown in \cite{degondfrouvelleliu15}, the differential equation satisfied by $J_f$ in the Vicsek case involves the spherical harmonics of degree 2 and higher of $f$; here the equation for $J_f$ is closed).\end{rem}

In particular the gradient-flow structure \eqref{ODEG} implies that the system \eqref{ODEG} will converge towards an equilibrium. When all the equilibria of the dynamical system \eqref{ODED} are hyperbolic the convergence towards the stable equilibria is exponentially fast (see \cite[Section 9.3]{hirsch2012differential}). The goal of the next section is to find which equilibria among the ones found in Section \ref{paragraphequilibria} are stable and to completely describe the asymptotic behaviour of the system depending on the initial condition and the local density $\rho$. We will see that phase transitions appear between ordered and disordered dynamics when the density $\rho$ increases. 

\subsection{Stability of the equilibria and conclusion}\label{paragraphstability}

Since the flow of \eqref{ODEDR3} is the image by the isomorphism $\diag^{-1}$ of the flow of \eqref{ODED}, an equilibria $D^\mathrm{eq}$ of \eqref{ODED} is stable (resp. unstable) if and only if $\diag^{-1}(D^\mathrm{eq})$ is a stable (resp. unstable) equilibria of \eqref{ODEDR3}. The stability properties of the equilibria of \eqref{ODEDR3} are much simpler to study than the stability properties of the equilibria of \eqref{ODED} since they are given by the signature of the Hessian matrix  $\Hess \widehat{V}(\widehat{D}^\mathrm{eq})\in\mathscr{S}_3(\R)$ of the potential $\widehat{V}$ given in equation \eqref{potentialR3} (the linearisation of ODE \eqref{ODEDR3} around an equilibrium $\widehat{D}^\mathrm{eq}$ is indeed $\frac{d}{dt}\widehat{H}=-\Hess \widehat{V}(\widehat{D}^\mathrm{eq})\widehat{H}$). In particular, an equilibrium $\widehat{D}^\mathrm{eq}$ is stable if and only if the signature of $\Hess \widehat{V}(\widehat{D}^\mathrm{eq})$ is $(+++)$.\\

Note that in the matrix framework \eqref{ODEG}, the Hessian of the potential \eqref{potential} in the Euclidean space $\MM_3(\R)$ would be a rank 4 tensor. Here we are reduced to the computation of the signature of $3\times3$ symmetric matrices. For a diagonal matrix $D\in\mathscr{D}_3(\R)$ and $\widehat{D}=\diag^{-1}(D)$ we will write with a slight abuse of notations:
\[\Hess V(D)\equiv\Hess\widehat{V}(\widehat{D}).\] 
When $D\in\mathscr{D}_3(\R)$ is an equilibrium of \eqref{ODED} we call signature of the Hessian matrix $\Hess V(D)$ the signature of $\Hess\widehat{V}(\widehat{D})$ where $\widehat{D}=\diag^{-1}(D)$. A simple computation shows that the Hessian matrix $\Hess V(D)$ is given by : 
\[\Hess V(D)=I_3-\frac{1}{2}\rho\Gamma,\]
where $\Gamma=(\Gamma_{ij})_{i,j}$ with : 
\[\Gamma_{ij}=\langle a_{ii} a_{jj}\rangle_{M_D}-\langle a_{ii}\rangle_{M_D}\langle a_{jj}\rangle_{M_D}.\]
The following theorem is an extension of Corollary \ref{phasediagramcorollary} and gives a full description of the equilibria of the BGK operator with their stability. 
\begin{theorem}[Stability of the equilibria of the ODE \eqref{ODED}]\label{equilibriaBGKoperatorstability}\
\begin{enumerate}[$\bullet$]
\item For $0<\rho<\rho^*$, the only equilibrium is $D=0$. This equilibrium is stable.
\item For $\rho^*<\rho<\rho_c$, the equilibrium $D=0$ and the equilibria of type (b) with parameter $\alpha_+$  are stable. The equilibria of type (b) with parameter $\alpha_-$ are unstable and the signature of the Hessian matrix is $(-++)$.
\item For $\rho>\rho_c$, the stable equilibria are the equilibria of type (b) with parameter $\alpha_1$. The other equilibria ($D=0$, type (b) with parameter $\alpha_3$ and type (c) with parameter $\pm\alpha_2$) are unstable and the signatures of the Hessian matrix are respectively $(---)$, $(+--)$ and $(++-)$.
\end{enumerate}
\end{theorem}

The proof of Theorem \ref{equilibriaBGKoperatorstability} will be based on the following lemma which states an important orbital invariance principle for the signature of the Hessian matrix.
\begin{lemma}[Orbital invariance of the signature]\label{orbitalinvariancesignature} Let $\widehat{D}\in\R^3$ be an equilibrium of \eqref{ODEDR3}. The signature of $\Hess \widehat{V}(\widehat{D})$ depends only on the type (a), (b) or (c) of $\diag(\widehat{D})$ as defined in Proposition \ref{necessarycondition}.
\end{lemma}
\begin{proof} Let $\widehat{D}_1$ and $\widehat{D}_2$ be two equilibria of \eqref{ODEDR3} such that $\diag(\widehat{D}_1)$ and $\diag(\widehat{D}_2)$ are of the same type. Then there exist $P,Q\in SO_3(\R)$ and $R\in\OO_3(\R)$ such that 
\[\widehat{D}_2=\diag^{-1}\left(P\diag(\widehat{D}_1)Q\right)=R\widehat{D}_1.\]
Moreover for all $M\in\MM_3(\R)$, the matrices $P,Q,R$ satisfy (permutation of the diagonal coefficients):
\[\left(
\begin{array}{c}
(PMQ)_{11}\\
(PMQ)_{22}\\
(PMQ)_{33}
\end{array} 
\right)
=R
\left(
\begin{array}{c}
M_{11}\\
M_{22}\\
M_{33}
\end{array}
\right)\]
where $M_{ii}$ is the $(i,i)$-th coefficient of the matrix $M$. Therefore, one can check that:
\[\nabla\widehat{V}(\widehat{D}_2)=R\nabla\widehat{V}(\widehat{D}_1)\]
and
\[\Hess\widehat{V}(\widehat{D}_2)=R\Hess\widehat{V}(\widehat{D}_1)R^T.\]
The conclusion follows from Sylvester's law of inertia (\cite[Chapter 8 Theorem 1]{laxlinearalgebra}).\end{proof}

\begin{proof}[\sc Proof (of Theorem \ref{equilibriaBGKoperatorstability})] Thanks to lemma \ref{orbitalinvariancesignature}, we therefore do not have to compute the signatures of the Hessian matrix taken at all equilibria, it is enough to choose one matrix in each orbit: for the equilibria of type (b) (see Proposition \ref{necessarycondition}) we will compute the signature of $\Hess(\alpha I_3)$ where $\alpha=\rho c_1(\alpha)$ and for the equilibria of type (c) we will compute the signature of $\Hess(\alpha \diag(1,0,0))$ where $\alpha=\rho c_2(\alpha)$. We first start with the case of the uniform equilibrium.\\

\noindent\textbf{\textit{Case 1. Uniform equilibrium $D=0$}}\\

For the uniform equilibrium $D=0$, $\langle a_{ij}\rangle=0$ for all $(i,j)$. Moreover, by the change of variable $A'= D^{ik}A$ where $k\ne i,j$, we have when $i\ne j$ : 
\[\langle a_{ii}a_{jj}\rangle=-\langle a_{ii}a_{jj}\rangle=0.\]
It proves that $\Gamma$ is diagonal. Then with the changes of variables $A'=P^{ij}A$ or $A'=AP^{ij}$ it can be seen that all the $3^2$ quantities $\langle a_{ij}^2\rangle$ are equal. Since their sum is equal to $n=3$ we get that 
\[\Hess V(0)=\left(1-\frac{1}{2}\rho\langle a_{11}^2\rangle\right)I_3= \left(1-\frac{\rho}{\rho_c}\right)I_3,\]
where $\rho_c=6$. In conclusion, the signature of $\Hess V(0)$ is $(+++)$ if $\rho<\rho_c$ and $(---)$ if $\rho>\rho_c$. \\

\noindent\textit{\textbf{Case 2. Equilibria of type (b) : $D=\alpha I_3$}}\\

Let $D=\alpha I_3$ with $\alpha=\rho c_1(\alpha)$. We have $c_1(\alpha)=\langle a_{11}\rangle_{M_D}=\langle a_{22}\rangle_{M_D}=\langle a_{33}\rangle_{M_D}$ and a change of variable of the type $A'=P^{ij}A(P^{ij})^T$ shows that all the $\langle a_{kk} a_{\ell\ell}\rangle_{M_D}$ are equal. The Hessian matrix is therefore equal to : 
\[\Hess V(\alpha I_3)=I_3-\frac{1}{2}\rho\left(\begin{array}{ccc}
\nu & \gamma & \gamma \\ 
\gamma & \nu & \gamma \\ 
\gamma & \gamma &\nu
\end{array}\right),\]
with $\nu=\langle a_{11}^2\rangle_{M_D}-\langle a_{11}\rangle_{M_D}^2$ and $\gamma=\langle a_{11}a_{22}\rangle_{M_D}-\langle a_{11}\rangle_{M_D}\langle a_{22}\rangle_{M_D}$. The eigenvalues of $\Hess V(D)$ are : 
\begin{enumerate}[$\bullet$]
\item $1-\frac{1}{2}\rho\nu-\rho\gamma$ of order 1 with eigenvector $(1,1,1)^T$. But taking the derivative with respect to $\alpha$ of \eqref{defc1} with $\Lambda=I_3$ we obtain the relation : 
\[c_1'(\alpha)=\frac{1}{2}\langle a_{11}^2\rangle_{M_D}+\langle a_{11}a_{22}\rangle_{M_D}-\frac{3}{2}\langle a_{11}\rangle_{M_D}^2=\frac{1}{2}\nu+\gamma,\]
and using $\rho=\alpha/c_1(\alpha)$ we can rewrite : 
\[1-\frac{1}{2}\rho\nu-\rho\gamma=1-\frac{\alpha c_1'(\alpha)}{c_1(\alpha)}=c_1(\alpha)\left(\frac{id}{c_1}\right)'(\alpha).\]
Its sign is then given by Proposition \ref{alphasurc} and the fact that $c_1(\alpha)$ has the same sign as $\alpha$~: $c_1(\alpha)\left(\frac{id}{c_1}\right)'(\alpha)>0$ when $\alpha<0$, $c_1(\alpha)\left(\frac{id}{c_1}\right)'(\alpha)<0$ when $0\leq\alpha<\alpha^*$ and $c_1(\alpha)\left(\frac{id}{c_1}\right)'(\alpha)>0$ when $\alpha>\alpha^*$.
\item $1-\frac{1}{2}\rho\nu+\frac{1}{2}\rho\gamma$ of order 2 with eigenvectors $(1,-1,0)^T$ and $(0,1,-1)^T$. It can be rewritten as:
\[1-\frac{1}{2}\rho\nu+\frac{1}{2}\rho\gamma=1-\frac{\alpha}{4c_1(\alpha)}\big\langle(a_{11}-a_{22})^2\big\rangle_{M_D}.\]
To determine this sign, we use the explicit volume form of $SO_3(\R)$ given by Rodrigues' formula \eqref{rodrigue} to see that : 
\[\frac{\alpha}{4c_1(\alpha)}\Big\langle(a_{11}-a_{22})^2\Big\rangle_{M_D} = \frac{\alpha}{5}\cdot\frac{\int_0^\pi \sin^2(\theta/2)(1-\cos\theta)^2e^{\alpha\cos\theta}\,d\theta}{\int_0^\pi\sin^2(\theta/2)(1+2\cos\theta)e^{\alpha\cos\theta}\,d\theta}.
\]

\begin{lemma}\label{complique} The function 
\[f:x\longmapsto 1- \frac{x}{5}\cdot\frac{\int_0^\pi \sin^2(\theta/2)(1-\cos\theta)^2e^{x\cos\theta}\,d\theta}{\int_0^\pi\sin^2(\theta/2)(1+2\cos\theta)e^{x\cos\theta}\,d\theta},\]
satisfies $f(0)=0$, $f(x)\geq0$ if $x\geq0$ and $f(x)\leq0$ if $x\leq0$\end{lemma}

\begin{proof} The value of $f(0)$ is given by the expansion of $\exp$. Note that : 
\[\frac{d}{d\theta}\Big\{\sin^2(\theta/2)\sin\theta\Big\}=\sin^2(\theta/2)(1+2\cos\theta),\]
so that an integration by parts shows : 
\[\int_0^\pi\sin^2(\theta/2)(1+2\cos\theta)e^{x\cos\theta}\,d\theta=x\int_0^\pi\sin^2(\theta/2)\sin^2(\theta)e^{x\cos\theta}d\theta.\]
We get that : 
\[f(x)=1-\frac{1}{5}\frac{\int_0^\pi \sin^2(\theta/2)(1-\cos\theta)^2e^{x\cos\theta}\,d\theta}{\int_0^\pi\sin^2(\theta/2)\sin^2(\theta)e^{x\cos\theta}d\theta}.\]
We have : 
\[f(x)\geq0\Longleftrightarrow \int_0^\pi \sin^2(\theta/2) \left(\frac{1}{5}(1-\cos\theta)^2-\sin^2(\theta)\right)e^{x\cos\theta}\,d\theta\leq0.\]
Linearizing $\sin^2(\theta/2)$ and expanding everything gives : 
\[\sin^2(\theta/2) \left(\frac{1}{5}(1-\cos\theta)^2-\sin^2(\theta)\right)=\frac{-1}{5}(3\cos\theta+2)(1-\cos\theta)^2,\]
so that : 
\[f(x)\geq0\Longleftrightarrow \int_0^\pi (3\cos\theta+2)(1-\cos\theta)^2 e^{x\cos\theta}\,d\theta\geq0.\]
Now for $x\geq0$, let $\theta_0=\arccos(-2/3)$. We cut the integral at $\theta_0$ and we get 
\[\int_0^{\theta_0} (3\cos\theta+2)(1-\cos\theta)^2 e^{x\cos\theta}\,d\theta\geq e^{-\frac{2}{3}x}\int_0^{\theta_0} (3\cos\theta+2)(1-\cos\theta)^2\,d\theta,\]
since the integrand is nonnegative and $x\geq0$. Similarly when the integrand is nonpositive 
\[\int_{\theta_0}^\pi (3\cos\theta+2)(1-\cos\theta)^2 e^{x\cos\theta}\,d\theta \geq e^{-\frac{2}{3}x}\int_{\theta_0}^\pi (3\cos\theta+2)(1-\cos\theta)^2 \,d\theta.\]
And finally : 
\[\int_0^\pi (3\cos\theta+2)(1-\cos\theta)^2 e^{x\cos\theta}\,d\theta \geq e^{-\frac{2}{3}x}\int_0^\pi (3\cos\theta+2)(1-\cos\theta)^2\,d\theta=0.\]
We find similarly that $f(x)\leq0$ when $x\leq0$.
\end{proof}
And therefore, we can deduce the sign of the eigenvalue: $1-\frac{\alpha}{4c_1(\alpha)}\big\langle(a_{11}-a_{22})^2\big\rangle_{M_D}>0$ when $\alpha>0$ and $1-\frac{\alpha}{4c_1(\alpha)}\Big\langle(a_{11}-a_{22})^2\big\rangle_{M_D}<0$ when $\alpha<0$.
 \end{enumerate}\
 
 \noindent\textit{\textbf{Case 3. Equilibria of type (c) : $D=\alpha\diag(1,0,0)$}}\\

For $D=\alpha\diag(1,0,0)$ with $\alpha=\rho c_2(\alpha)$, using the parametrisation \eqref{parametrisationSOSphere}, it holds that :
\begin{align*}
&\int_{SO_3(\R)}a_{22}e^{\frac{\alpha}{2}a_{11}}\,dA=\int_{SO_3(\R)}a_{22}e^{\frac{\alpha}{2}a_{33}}\,dA\\
&\hspace{1cm}=\frac{1}{8\pi^2}\int_{\theta=0}^{2\pi}\int_{\phi_2=0}^\pi e^{\frac{\alpha}{2}\cos\phi_2}\sin\phi_2\int_{\phi_1=0}^{2\pi}(-\sin\theta\sin\phi_1+\cos\theta\cos\phi_1\cos\phi_2)d\phi_1d\phi_2d\theta \\
&\hspace{1cm}=0.
\end{align*}
where the first equality comes from the change of variable $A\mapsto P^{13}A(P^{13})^T$. Similarly,
\[\langle a_{22}\rangle_{M_D}=\langle a_{33}\rangle_{M_D}=\langle a_{11} a_{22}\rangle_{M_D}=\langle a_{11} a_{33}\rangle_{M_D}=0.\]
The Hessian matrix is therefore equal to : 
\[I_3-\frac{1}{2}\rho \left(\begin{array}{ccc}
\langle a_{11}^2\rangle_{M_D}-\langle a_{11}\rangle_{M_D}^2 &0&0 \\ 
0 & \langle a_{22}^2\rangle_{M_D} & \langle a_{22}a_{33}\rangle_{M_D} \\
0 & \langle a_{22}a_{33}\rangle_{M_D} & \langle a_{33}^2\rangle_{M_D}
\end{array}\right),\]
and since $\langle a_{22}^2\rangle_{M_D}=\langle a_{33}^2\rangle_{M_D}$ as it can be seen with the change of variable $A\mapsto P^{23}A(P^{23})^T$, its eigenvalues are : 
\[1-\frac{1}{2}\rho\Big(\langle a_{11}^2\rangle_{M_D}-\langle a_{11}\rangle_{M_D}^2\Big),\]
with eigenvector $(1,0,0)^T$, 
\[1-\frac{1}{2}\rho\Big( \langle a_{22}^2\rangle_{M_D}- \langle a_{22}a_{33}\rangle_{M_D}\Big),\]
with eigenvector $(0,1,-1)^T$ and
\[1-\frac{1}{2}\rho\Big( \langle a_{22}^2\rangle_{M_D}+ \langle a_{22}a_{33}\rangle_{M_D}\Big),\]
with eigenvector $(0,1,1)^T$.
We have as before : 
\[c_2'(\alpha)=\frac{1}{2}\Big(\langle a_{11}^2\rangle_{M_D}-\langle a_{11}\rangle_{M_D}^2\Big),\]
so the first eigenvalue can be rewritten 
\[1-\frac{1}{2}\rho\Big(\langle a_{11}^2\rangle_{M_D}-\langle a_{11}\rangle_{M_D}^2\Big)=1-\frac{\alpha c_2'(\alpha)}{c_2(\alpha)}=c_2(\alpha)\left(\frac{id}{c_2}\right)'(\alpha)>0,\]
where we have used Proposition \ref{alphasurc} to determine the sign. The  two other eigenvalues are equal to: 
\[1-\frac{1}{4}\rho \Big\langle (a_{22}\pm a_{33})^2\Big\rangle_{M_D}\]
Using the change of variable $A\mapsto P^{13}A(P^{13})^T$ and the parametrisation \eqref{parametrisationSOSphere}, one can see that : 
\begin{align*}
&\int_{SO_3(\R)} (a_{22}\pm a_{33})^2 e^{\frac{\alpha}{2}a_{11}}\,dA=\int_{SO_3(\R)} (a_{22}\pm a_{11})^2 e^{\frac{\alpha}{2}a_{33}}\,dA\\
&=\frac{1}{8\pi^2}\int_{\theta=0}^{2\pi}\int_{\phi_1=0}^{2\pi}\cos^2(\phi_1\pm\theta)\int_{\phi_2=0}^{\pi}\sin\phi_2\,(1\pm\cos\phi_2)^2 e^{\frac{\alpha}{2}\cos\phi_2}\,d\phi_2\,d\phi_1\,d\theta
\end{align*}
so that : 
\[\Big\langle (a_{22}\pm a_{33})^2\Big\rangle_{M_D}=\frac{1}{2}\Big[(1\pm\cos\phi)^2\Big]_{\alpha}\]
where $[\cdot]_\alpha$ is defined in Proposition \ref{c1c2}.
Using the relation $\rho=\frac{\alpha}{c_2(\alpha)}=\frac{4}{[ \sin^2\phi]_\alpha}$ (see Formula \eqref{PierreN}), the two eigenvalues are equal to : 
\[1-\frac{1}{2}\frac{\int_0^\pi \sin\phi(1\pm\cos\phi)^2e^{\frac{\alpha}{2}\cos\phi}\,d\phi}{\int_0^\pi \sin^3\phi e^{\frac{\alpha}{2}\cos\phi}\,d\phi}.\]
With the same technique used in the previous paragraph (Lemma \ref{complique}) it is possible to show that the eigenvalue $1-\frac{1}{4}\rho \Big\langle (a_{22}- a_{33})^2\Big\rangle_{M_D}$ is nonpositive when $\alpha<0$ and nonnegative when $\alpha>0$. The contrary holds for the other eigenvalue (nonnegative when $\alpha<0$ and nonpositive when $\alpha>0$). Finally the signs of these two eigenvalues are always $(+-)$. \\

The conclusion of the proof follows from the study of the roots of Equations \eqref{compatibilityc1} and \eqref{compatibilityc2} provided by Corollary \ref{phasediagramcorollary} and depicted in Figure \ref{phasediagram}. In particular, the cases $\alpha=\alpha^*$ and $\alpha=0$ correspond to the critical cases $\rho=\rho^*$ and $\rho=\rho_c$ which are studied in the next section.
\end{proof}

\begin{rem} 
The above calculations are similar to the computation of the equilibria and their stability of
\[\Psi~: D\in\diag(\mathscr{T})\mapsto \mathcal{F}[M_D]\]
where $\mathcal{F}$ is the free-energy \eqref{freeenergy}. In particular, it can be shown that the equilibria of $\Psi$ and their stability are the same as the ones described in Theorem \ref{equilibriaBGKoperatorstability}. In particular, since \eqref{freeenergy} is also a free energy in the Fokker-Planck case, this analysis shows the instability of equilibria of the Fokker-Planck of the form $\rho M_{D}$ when $D$ is of one of the unstable equilibria of \eqref{ODED} described in Theorem \ref{equilibriaBGKoperatorstability}. This technique is similar to the one which was used in \cite{wangzhou11} in the case of 3D polymers. However, it does not provide global or local convergence of the solution of the Fokker-Planck equation towards an equilibrium. This requires further investigations which will be left for future work. In the Vicsek case \cite{degondfrouvelleliu15}, it was based on a LaSalle's principle and on estimates for the dissipation term. 
\end{rem}

We can finally prove Theorem \ref{theoremconvergence}~:
\begin{proof}[\sc Proof (of Theorem \ref{theoremconvergence})] Thanks to the Duhamel's formula \eqref{duhamel}, the asymptotic behaviour of $f(t)$ as $t\to+\infty$ is given by the asymptotic behaviour of $J_{f(t)}$. Thanks to Proposition \ref{reductionODE} and Corollary \ref{ODEssvd}, we only have to study the asymptotic behaviour of the solution of the ODE \eqref{ODED} where the initial condition is the diagonal part of a SSVD of $J_{f_0}$. Since this equation has a gradient-flow structure \eqref{ODEG}, we know that the solution $D(t)$ converges as $t\to+\infty$ towards an equilibrium $D^\mathrm{eq}$ and consequently $J_{f(t)}\to PD^\mathrm{eq}Q:=J^{\mathrm{eq}}$. Moreover the equilibrium $D^\mathrm{eq}$ is a stable equilibrium provided that $D_0$ does not belong to the stable manifold of an unstable equilibrium. Since these manifolds have dimension at most 2, the union of these manifolds, called $\mathcal{N}_\rho$ is of zero measure and there is convergence towards a stable equilibrium for all $D_0\notin\mathcal{N}_\rho$. In this case, the convergence is locally exponentially fast in the sense that there exist constants $\delta,\lambda,C>0$ such that if $\|J_{f_0}-J^\mathrm{eq}\|\leq\delta$ then for all $t>0$,
\begin{equation}\label{expconvergenceode}\|J_{f(t)}-J^\mathrm{eq}\|\leq Ce^{-\lambda t}.\end{equation}
Let $f^\mathrm{eq}=\rho M_{J^\mathrm{eq}}$. It follows from Duhamel's formula \eqref{duhamel} that for all $A\in SO_3(\R)$~:
\begin{equation}\label{duhameleq}|f(t,A)-f^\mathrm{eq}(A)|\leq e^{-t}|f_0(A)-f^\mathrm{eq}(A)|+\rho e^{-t}\int_0^t e^s|M_{J_{f(s)}}(A)-M_{J^\mathrm{eq}}(A)|ds.\end{equation}
Since $SO_3(\R)$ is compact and $J_{f(t)}$ is bounded uniformly in $t$, there exists a constant $L>0$ such that for all $t>0$, the following Lipschitz bound holds:
\begin{equation}\label{lipschitzM}\forall A\in SO_3(\R),\,\,\,|M_{J_{f(t)}}(A)-M_{J^\mathrm{eq}}(A)|\leq L\|J_{f(t)}-J^\mathrm{eq}\|.\end{equation}
Reporting \eqref{lipschitzM} into \eqref{duhameleq} and using \eqref{expconvergenceode}, the first point of Theorem \ref{theoremconvergence} follows with constants $K=CL/|1-\lambda|$ and $\mu=\min(1,\lambda)$ when $\lambda\ne1$ and $K=CL$ and $\mu=1-\varepsilon$ for any $\varepsilon>0$ when $\lambda=1$.\\

The stability of the equilibria of the dynamical system \eqref{ODED} is given in Theorem \ref{equilibriaBGKoperatorstability}. Finally the conclusions of points 2.(ii) and 2.(iii) follow from the fact that the diagonal parts of the SSVD of the equilibria of type (b) with parameters $\alpha_+$ and $\alpha_1$ are respectively $\alpha_+ I_3$ and $\alpha_1 I_3$.\end{proof}

\subsection{Final remark: characterisation of the stable manifold of the unstable equilibria and critical cases}\label{conclusioncompliquee}

This section is devoted to a more precise description of the subset $\mathscr{N}_\rho$ of zero measure of initial conditions which do not necessarily lead to one of the behaviours detailed in the previous section (see Theorem \ref{theoremconvergence}).\\

When $\rho\ne\rho^*$ and $\rho\ne\rho_c$, all the equilibria of \eqref{ODEDR3} are hyperbolic and we can apply the stable manifold theorem (\cite{perko2013differential} Section 2.7) which states that for a given hyperbolic equilibrium $\widehat{D}^\mathrm{eq}$ of the nonlinear equation \eqref{ODEDR3}, the stable set of $\widehat{D}^\mathrm{eq}$ is a smooth manifold. Its dimension is given by the number of minuses in the signature of $-\Hess \widehat{V}(\widehat{D}^\mathrm{eq})$ and its tangent space at $\widehat{D}^\mathrm{eq}$ is the stable subspace of the linearized system around $\widehat{D}^\mathrm{eq}$. We are now ready to describe the behaviour of the system \eqref{ODEDR3} for a given density $\rho$ and an initial condition $\widehat{D}_0\in\R^3$. Note that we can restrict ourselves to the case $\widehat{D}_0\in\diag^{-1}(\mathscr{D})$ and that our goal is to describe $\diag^{-1}(\mathscr{N}_\rho)$ (in the $\R^3$-framework). We write $D(t)=\diag\widehat{D}(t)$ the solution of \eqref{ODED} with initial condition $D_0=\diag\widehat{D}_0$. \\

\noindent\textbf{Case 1. When $\rho<\rho^*$}\\

When $\rho<\rho^*$, there is only the uniform equilibrium : $D(t)\underset{t\to+\infty}{\longrightarrow}0$ and $\mathscr{N}_\rho=\emptyset$. \\

\noindent\textbf{Case 2. When $\rho^*<\rho<\rho_c$} \\

There are three equilibria for \eqref{ODEDR3} : two are stable ($0$ and $\alpha_+ (1,1,1)^T$) and one is unstable ($\alpha_-(1,1,1)^T$) which satisfies that $-\Hess(\widehat{V})(\alpha_- (1,1,1)^T)$ has signature $(--+)$. The stable manifold theorem ensures that $\diag^{-1}(\mathscr{N}_\rho)$ is the stable set of the unstable equilibrium and is a smooth manifold of dimension 2 and its tangent plane at $\alpha_-(1,1,1)^T$ admits the eigenvectors $(1,-1,0)^T$ and $(0,1,-1)^T$ for basis. The unstable manifold of the unstable equilibrium has dimension : it is the line of direction $(1,1,1)^T$. \\ 

Finally, the space $\R^3$ is partitioned in two domains by $\diag^{-1}(\mathscr{N}_\rho)$ : depending on which domain $\widehat{D}_0$ belongs to in $\R^3\setminus\diag^{-1}(\mathscr{N}_\rho)$, $\widehat{D}(t)$ will converge towards either the uniform equilibrium or to $\alpha_+(1,1,1)^T$ (or towards $\alpha_- (1,1,1)^T$ if $\widehat{D}_0$ belongs to its stable manifold). \\

\noindent\textbf{Case 3. When $\rho>\rho_c$}\\

When $\rho>\rho_c$, $\widehat{D}(t)$ can converge towards one of the four equilibria : $0$, $\alpha_3(-1,-1,1)^T$, $\alpha_2(1,0,0)^T$ or $\alpha_1 (1,1,1)^T$ (see Figure \ref{phasediagram}).
\begin{enumerate}[$\bullet$]
\item The uniform equilibrium is unstable and $-\Hess(\widehat{V})(0)$ has signature $(+++)$ therefore, there is no stable direction and $D(t)$ cannot converge towards 0 unless $D_0=0$ (and then $D(t)=0$ for all $t\in\R_+$).
\item The equilibrium $\alpha_3\diag(-1,-1,1)$ is unstable and its stable manifold $\mathcal{M}_1$ has dimension 1: it is the half line $\R_+^*\left(\begin{array}{c}1\\1\\-1\end{array}\right)$ (it is indeed an invariant manifold of dimension 1 and a solution on this half-line cannot converge towards an other equilibrium). Therefore $D(t)$ converges towards $\alpha_3\diag(-1,-1,1)$ if and only if $\widehat{D}_0$ lies on this half line.
\item The equilibrium $\alpha_2(1,0,0)^T$ is unstable and its stable manifold has dimension 2 with a tangent plane generated by the two vectors $(1,0,0)^T$ and $(0,1,-1)^T$ : it is the plane $\{d_2+d_3=0\}$. We note that this plane is an invariant manifold, so if $\widehat{D}_0$ belongs to this plane, the limit when $t\to+\infty$ also belongs to this plane. But since the half-lines $\R_+^*\left(\begin{array}{c}1\\1\\-1\end{array}\right)$ and $\R_+^*\left(\begin{array}{c}1\\-1\\1\end{array}\right)$ are the stable manifolds of the equilibria $\alpha_3(-1,-1,1)^T$ and $\alpha_3(-1,1,-1)^T$, the limit when $t\to+\infty$ is $\alpha_2(1,0,0)^T$ if and only if $D_0$ belongs to the (open) quarter of the plane $\{d_2+d_3=0\}$ delimited by these two half-lines: this is the stable manifold of $\alpha_2\diag(1,0,0)$. In conclusion, $D(t)$ converges towards $\alpha_2\diag(1,0,0)$ if and only if $\widehat{D}_0$ is of the form $(d_1,d_2,-d_2)^T$ with $d_1>d_2\geq0$. Since $\widehat{D}_0\in\diag^{-1}(\mathscr{D})$, we can restrict ourselves to the eighth of plan $\mathcal{M}_2$ delimited by the half lines $\R_+^*(1,1,-1)^T$ excluded and $\R_+^*(1,0,0)^T$ included.
\item In every other cases, $D(t)$ converges towards $\alpha_1 I_3$ which is the only stable equilibrium.
\end{enumerate}

Finally, in the case $\rho>\rho_c$, the subset $\mathscr{N}_\rho\cap\mathscr{D}$ is equal to $\diag\big(\{0\}\cup\mathcal{M}_1\cup\mathcal{M}_2\big)$.\\

\noindent\textbf{The critical cases}\label{criticalcases}\\

We end this section by an informal description of the expected behaviour in the critical cases. 

\begin{enumerate}[$\bullet$]
 \item When $\rho=\rho^*$ the uniform equilibrium is stable and there is an other equilibrium of the form $\alpha^* I_3$ with $\alpha^*>0$. This equilibrium is non hyperbolic: the kernel of $-\Hess \widehat{V}(\alpha^* (1,1,1)^T)$ is one dimensional, spanned by the vector $(1,1,1)^T$. The two other eigenvalues are negative. Therefore, for the system \eqref{ODEDR3}, there exists a center manifold of dimension 1 and a stable manifold of dimension 2, which tangent plane being the orthogonal of $(1,1,1)^T$.\\
 
 If $\widehat{D}_0$ belongs to the stable manifold, $D(t)$ converges exponentially fast towards $\alpha^* I_3$ (\cite{haragus2010local} Theorem 3.22). The stable manifold of dimension 2 delimits two domains and one of them is included in the subset $\{x\in\R^3, x\cdot(1,1,1)^T\geq\alpha^*\}$. If $\widehat{D}_0$ belongs to this domain, then it belongs to a center manifold and $\widehat{D}(t)$ is attracted exponentially fast towards the the line $\R(1,1,1)^T$ and converges at rate $1/t$ towards $\alpha^* (1,1,1)$. In every other cases, $\widehat{D}(t)$ converges exponentially fast towards 0. \\ 
 
 \item When $\rho=\rho_c$ there is one stable equilibrium of the form $\alpha I_3$ with $\alpha>0$. The uniform equilibrium is non hyperbolic and $-\Hess \widehat{V}(0)=0$.\\
 
In the case when $\widehat{D}(t)$ converges towards 0, with a formal computation, the rate of convergence is expected to be either $1/t$ or $1/\sqrt{t}$ depending on $\widehat{D}_0$. The convergence towards the stable anisotropic equilibrium is exponentially fast. 
\end{enumerate}
\section{Macroscopic limit for the stable equilibria}\label{macroscopic}

We now go back to the spatially inhomogeneous model. We want to investigate (at least formally) the hydrodynamic models derived from the BGK equation \eqref{bgk}. To do so, we introduce the scaling $t'=\varepsilon t$ and $x'=\varepsilon x$ for $\varepsilon>0$ and we define $f^\varepsilon(t',x',A):=f(t, x, A)$. After this change of variables in \eqref{bgk} and dropping the primes, we see that $f^\varepsilon$ satisfies the following equation:
\begin{equation}\label{bgkepsilon}
\partial_t f^\varepsilon+(Ae_1\cdot \nabla_x)f^\varepsilon = \frac{1}{\varepsilon}\left(\rho_{f^\varepsilon}M_{J_{f^\varepsilon}}-f^\varepsilon\right).
\end{equation}
We want to investigate the macroscopic limit $\varepsilon\to 0$ with the assumption that $f^\varepsilon$ converges towards a stable equilibrium. Thanks to the results of the last section, we will assume that $f^\varepsilon\to\rho M_{\alpha\Lambda}$ where $\rho=\rho(t,x)$, $\alpha=\rho c_1(\alpha)$, $\alpha\in\R_+$ and $\Lambda=\Lambda(t,x)\in\MM_3(\R)$ (with a notion of convergence as strong as needed). Since the equilibrium is assumed to be stable there are two cases: either $\Lambda\in SO_3(\R)$ (and therefore $\rho>\rho^*$) or $\Lambda=0$ that is to say $f^\varepsilon$ is uniform in the body-orientation variable and converges towards $\rho=\rho(t,x)$ (and therefore $\rho<\rho_c$). For a given time $t\in\R_+$, we will say that $x\in\R^3$ belongs to a disordered region when $\Lambda(t,x)=0$. Otherwise, when $\Lambda(t,x)\in SO_3(\R)$, we will say that $x\in\R^3$ belongs to an ordered region. \\

The purpose of the two next sections is to write at least formally the hydrodynamic equations satisfied by $\rho=\rho(t,x)$ and $\Lambda=\Lambda(t,x)$. First notice that integrating \eqref{bgkepsilon} over $SO_3(\R)$ leads to the conservation law: 
\begin{equation}\label{conservationrho}\partial_t \rho^\varepsilon+\nabla_x\cdot j[f^\varepsilon]=0,\end{equation} 
where $\rho^\varepsilon\equiv\rho_{f^\varepsilon}$ and 
\[j[f^\varepsilon]:=\int_{SO_3(\R)} Ae_1 f^\varepsilon\,dA\equiv J_{f^\varepsilon} e_1.\]
The macroscopic model then depends on the region considered. 
\begin{enumerate}
\item In a disordered region, $j[f^\varepsilon]\to0$ and assuming that the convergence is sufficiently strong, we get that $\partial_t\rho=0$. To obtain more information we will look at the next order in the Chapman-Enskog expansion (Section \ref{diffusionmodel}). 
\item In an ordered region, $j[f^\varepsilon]\to\alpha\Lambda e_1$ where $\alpha=\rho c_1(\alpha)$ and therefore, assuming that the convergence is strong enough:
\[\partial_t \rho+\nabla_x\cdot (\alpha\Lambda e_1)=0.\]
However due to the lack of conserved quantities, we will need specific tools to write an equation satisfied by $\Lambda=\Lambda(t,x)$ in order to obtain a closed system of equations on $(\rho,\Lambda)$. This is the purpose of Section \ref{orderedregion}.
\end{enumerate}
\subsection{Diffusion model in a disordered region}\label{diffusionmodel}

We consider a region where $f^\varepsilon$ converges as $\varepsilon\to0$ to a density $\rho(t,x)$ uniform in the body-attitude variable. The following proposition gives the diffusion model obtained by looking at the next order in the Chapman-Enskog expansion.

\begin{proposition}[Formal]\label{diffusionproposition}
In a disordered region, the density $\rho^\varepsilon$ satisfies formally at first order the following diffusion equation: 
\begin{equation}\label{diffusionrho}\partial_t\rho^\varepsilon=\varepsilon\nabla_x\cdot\left(\frac{\frac{1}{3}\nabla_x\rho^\varepsilon}{1-\frac{\rho^\varepsilon}{\rho_c}}\right),\,\,\,\,\,\,\rho_c=6.\end{equation}
\end{proposition}

\begin{proof} We follow the same calculations as in \cite{degondfrouvelleliu15}~: we write $f^\varepsilon=\rho^\varepsilon+\varepsilon f_1^\varepsilon$ (where $f^\varepsilon_1$ is defined by this relation) and notice that: 
\[J_{f^\varepsilon}=\varepsilon J_{f_1^\varepsilon}\,\,\,\,\,\text{and}\,\,\,\,\,\,M_{J_{f^\varepsilon}}(A)=1+\varepsilon J_{f_1^{\varepsilon}}\cdot A+\OO(\varepsilon^2).\]
Inserting this in \eqref{bgkepsilon}, multiplying by $A$ and integrating over $SO_3(\R)$ leads to: 
\begin{equation}\label{Jint}J_{f_1^\varepsilon}=\rho^\varepsilon\int_{SO_3(\R)} (J_{f_1^\varepsilon}\cdot A)A\,dA-\int_{SO_3(\R)} Ae_1\cdot\nabla_x\rho^\varepsilon AdA+\OO(\varepsilon).\end{equation}
Using Lemma \ref{JAA}, it holds that
\begin{equation}\label{Jint1}\int_{SO_3(\R)} (J_{f_1^\varepsilon}\cdot A)A\,dA=\frac{1}{6}J_{f_1^\varepsilon}.\end{equation}
To compute the second term, we note that $Ae_1\cdot\nabla_x\rho^\varepsilon=2A\cdot R_{\rho^\varepsilon}$ where $R_{\rho^\varepsilon}$ is the matrix, the first column of which is equal to $\nabla_x\rho^\varepsilon$ and the others are equal to zero. Using Lemma \ref{JAA} we obtain 
\begin{equation}\label{Jint2}\int_{SO_3(\R)} Ae_1\cdot\nabla_x\rho^\varepsilon AdA=\frac{1}{3}R_\rho^\varepsilon.\end{equation}
By multiplying \eqref{Jint} by $e_1$, it follows from \eqref{Jint1} and \eqref{Jint2} that : 
\[\left(1-\frac{\rho^\varepsilon}{\rho_c}\right)J_{f_1^{\varepsilon}}e_1=-\frac{1}{3}\nabla_x \rho^\varepsilon+\OO(\varepsilon),\]
which gives the result by inserting this in \eqref{conservationrho}.
\end{proof}

\begin{rem}
This analysis does not depend on the dimension. In $SO_n(\R)$ the same formal result holds: 
\[\partial_t\rho^\varepsilon=\varepsilon\nabla_x\cdot\left(\frac{\frac{1}{n}\nabla_x\rho^\varepsilon}{1-\frac{\rho^\varepsilon}{\rho_c}}\right),\,\,\,\,\,\,\rho_c=2n.\]
\end{rem}

\subsection{Self-organised hydrodynamics in an ordered region}\label{orderedregion}

In the following, for a given density $\rho\in\R_+$, $\alpha(\rho)$ denotes the maximal nonnegative root of $\alpha=\rho c_1(\alpha)$. We are going to prove the following theorem. 

\begin{theorem}[Formal]\label{theoremSOHB} We suppose that $f^\varepsilon\to\rho(x,t) M _{J(x,t)}$ (as strongly as necessary) as $\varepsilon\to0$ where $J(x,t)=\alpha(\rho(x,t))\Lambda(x,t)$ and $\Lambda(x,t)\in SO_3(\R)$. Then $\rho$ and $\Lambda$ satisfy the following system of partial differential equations: 
\begin{subequations}
\label{SOHB}
\begin{align}
&\partial_t \rho+\nabla_x\cdot (\rho c_1(\alpha(\rho))\Lambda e_1)=0 ,\label{SOHB1} \\
&\rho(\partial_t\Lambda+\tilde{c}_2((\Lambda e_1)\cdot\nabla_x)\Lambda)+\tilde{c_3}[(\Lambda e_1)\times\nabla_x\rho]_\times\Lambda \nonumber \\ 
&\hspace{4.2cm}+c_4\rho[-\mathbf{r}_x(\Lambda)\times(\Lambda e_1)+\delta_x(\Lambda)\Lambda e_1]_\times\Lambda=0. \label{SOHB2}
\end{align}
\end{subequations}
where $\tilde{c}_2$, $\tilde{c_3}$, $c_4$ are functions of $\rho$ to be defined later and $\delta$ and $\mathbf{r}$ are the ``divergence'' and ``rotational'' operators defined in \cite{degondfrouvellemerino17} : if $\Lambda(x)=\exp([\mathbf{b}(x)]_\times)\Lambda(x_0)$ with $\mathbf{b}$ smooth around $x_0$ and $\mathbf{b}(x_0)=0$, then
\[\delta_x(\Lambda)(x_0):=\nabla_x\cdot\mathbf{b}(x)|_{x=x_0}\,\,\,\text{and}\,\,\,\mathbf{r}_x(\Lambda)(x_0):=\nabla_x\times\mathbf{b}(x)|_{x=x_0},\]
where $\nabla_x\times$ is the curl operator. 
\end{theorem} 

The first equation \eqref{SOHB1} is the conservation law \eqref{conservationrho} and the goal is to obtain the equation \eqref{SOHB2} for $\Lambda=\Lambda(t,x)$. However, here and contrary to the classical gas dynamics, the total momentum is not conserved:
\[\frac{d}{dt}\int_{SO_3(\R)} f A\,dA\ne0\]
and we therefore cannot deduce easily a closed system of equations. This lack of conserved quantities is specific to self-propelled particle models such as the Vicsek model. The main tool to tackle the problem will be the Generalised Collision Invariants method introduced in \cite{degondmotsch08} for the study of the hydrodynamic limit of the continuum Vicsek model. Its precise setting in the context of our body-attitude model is detailed Section \ref{sectionGCI}. The formal proof of Theorem \ref{theoremSOHB} can be found in Section \ref{sectionhydrodynamiclimit}.

\subsubsection{Generalised collision invariants}\label{sectionGCI}

To obtain an equation on $\Lambda$, the main tool are the Generalised Collisional Invariants (GCI) first introduced in \cite{degondmotsch08}. For a given $J\in \MM_3(\R)$, we define first the linear collision operator: 
\[\mathcal{L}_J(f)=\rho_f M_J-f,\]
so that $Q_{BGK}(f)=\mathcal{L}_{J_f}(f)$. Let $J\in\MM_3(\R)$ with $\det J>0$ and let $\Lambda~:=PD(J)\in SO_3(\R)$ be the orthogonal part of its polar decomposition. The set of GCI associated to $J$ is defined as:
\begin{equation}\label{defGCI}\mathscr{C}_{J}~:=\left\{\psi~: SO_3(\R)\to\R,\,\,\,\int_{SO_3(\R)} \mathcal{L}_J(f)\,\psi\,dA=0\,\,\,\text{for all}\,\,f\,\,\text{such that}\,\,P_{T_\Lambda}(J_f)\right\}.\end{equation}
The condition $\psi\in \mathscr{C}_J$ is equivalent to: 
\[\int_{SO_3(\R)} f(\langle \psi\rangle_{M_J}-\psi)\,dA=0\,\,\,\,\,\,\text{for all}\,\,\, f\,\,\,\text{such that}\,\,\, P_{T_\Lambda}(J_f)=0.\]
Therefore, following the ideas of the proof of \cite[Proposition 4.3]{degondfrouvellemerino17}, we have: 
\[\psi\in\mathscr{C}_J\,\,\,\Longleftrightarrow\,\,\,\exists B\in T_{\Lambda},\,\,\,\langle \psi\rangle_{M_J}-\psi(A)=B\cdot A,\]
that is to say: 
\[\psi\in\mathscr{C}_{J}\,\,\,\Longleftrightarrow\,\,\,\exists B\in T_{\Lambda},\,\,\exists C\in\R,\,\,\,\psi(A)=-B\cdot A + C,\]
or equivalently since $B\in T_\Lambda$ means that there exists $P\in\mathscr{A}_3(\R)$ such that $B=\Lambda P$~:
\[\mathscr{C}_J = \Span\left(1,\underset{P\in\mathscr{A}_3(\R)}{\bigcup} \psi^\Lambda_P\right),\]
where
\[\psi_P^\Lambda(A)=-P\cdot\Lambda^T A.\]
Now for any $P\in\mathscr{A}_3(\R)$, denoting $\Lambda_{f^\varepsilon}\equiv \Lambda^\varepsilon=PD(J_{f^\varepsilon})$, we get by multiplying the equation \eqref{bgkepsilon} by $\psi_P^{\Lambda^\varepsilon}$~: 
\begin{equation}\label{equationGCI}\int_{SO_3(\R)} (\partial_t f^\varepsilon +Ae_1\cdot \nabla_x f^\varepsilon)P\cdot(\Lambda^\varepsilon)^TA\,dA=0.\end{equation}
The right-hand side vanishes by Definition \eqref{defGCI} of the GCI. Note that if $f^\varepsilon\to\rho M_J$ with $\det(J)>0$, we have also $\det(J_{f^\varepsilon})>0$ for $\varepsilon$ sufficiently small and $\Lambda^\varepsilon\in SO_3(\R)$.

\subsubsection{Hydrodynamic limit (formal proof of Theorem \ref{theoremSOHB})}\label{sectionhydrodynamiclimit}

Taking formally the limit $\varepsilon\to0$ in \eqref{equationGCI} and since it is true for all $P\in\mathscr{A}_3(\R)$, we obtain: 
\[X~:=\int_{SO_3(\R)} \Big(\partial_t(\rho M_{\alpha\Lambda})+Ae_1\cdot\nabla_x(\rho M_{\alpha\Lambda})\Big)(\Lambda^T A-A^T\Lambda)\,dA=0.\]
We have: 
\begin{eqnarray*}
\partial_t (M_{\alpha\Lambda})&=& \Big(\partial_t(\alpha\Lambda)\cdot A - \langle \partial_t(\alpha\Lambda)\cdot A\rangle_{M_{\alpha\Lambda}}\Big) M_{\alpha\Lambda} \\ 
&=& \alpha'\partial_t\rho (A\cdot\Lambda-\langle A\cdot \Lambda\rangle_{M_{\alpha\Lambda}})M_{\alpha\Lambda} + \alpha\partial_t \Lambda\cdot(A-\langle A\rangle_{M_{\alpha\Lambda}})M_{\alpha\Lambda},
\end{eqnarray*}
where $\alpha'$ denotes the derivative of $\alpha(\rho)$ with respect to $\rho$ and similarly for $\partial_i(M_{\alpha\Lambda})$. With this we compute the term: 
\begin{multline*}
(\partial_t+Ae_1\cdot\nabla_x)(\rho M_{\alpha\Lambda}) = \\
M_{\alpha\Lambda}(A)\left(1+\rho\alpha'\left(A\cdot \Lambda-\frac{3}{2}c_1(\alpha)\right)\right)(\partial_t+Ae_1\cdot\nabla_x)\rho \\
+ M_{\alpha\Lambda}(A)\rho\alpha A\cdot(\partial_t+Ae_1\cdot\nabla_x)\Lambda,
\end{multline*}
where we have used that:
\[\langle A\cdot\Lambda\rangle_{M_{\alpha\Lambda}} = \frac{3}{2}c_1(\alpha),\]
and
\[\Lambda\cdot(\partial_t+Ae_1\cdot\nabla_x)\Lambda=0.\]
Most of the terms that appear in $X$ are computed in \cite{degondfrouvellemerino17}. Precisely, 
\begin{equation}\label{XXX}X=X_1+X_2+X_3+X_4+Y\end{equation}
where $X_1$, $X_2$, $X_3$ and $X_4$ are computed in \cite{degondfrouvellemerino17}~:
\begin{align*}
&X_1:=\int_{SO_3(\R)}\partial_t\rho M_{\alpha\Lambda}(A)(\Lambda^TA-A^T\Lambda)\,dA=0,\\
&X_2:=\int_{SO_3(\R)} \alpha\rho(A\cdot\partial_t\Lambda) M_{\alpha\Lambda}(A)(\Lambda^T A-A^T\Lambda)\,dA=C_2\rho\alpha\Lambda^T\partial_t\Lambda,\\
&X_3:=\int_{SO_3(\R)} Ae_1\cdot\nabla_x\rho M_{\alpha\Lambda}(A)(\Lambda^TA-A^T\Lambda)\,dA=C_3[e_1\times\Lambda^T\nabla_x\rho]_\times,\\
&X_4:=\int_{SO_3(\R)} \rho\alpha\big(A\cdot (Ae_1\cdot\nabla_x)\Lambda\big)M_{\alpha\Lambda}(A)(\Lambda^TA-A^T\Lambda)\,dA\\
&\hspace{6cm}=\rho\alpha(C_4[Le_1]_\times+C_5[L^T e_1+\Tr(L)e_1]_\times),\end{align*}
where the coefficients
\[C_2=C_3:=\frac{2}{3}\{\sin^2\theta\}_\alpha,\,\,\,\,C_4:=\frac{2}{15}\{\sin^2\theta(1+4\cos\theta)\}_\alpha,\,\,\,\,\text{and}\,\,\,\,C_5:=\frac{2}{15}\{\sin^2\theta(1-\cos\theta)\}_\alpha,\]
and the matrix
\[L:=\Lambda^T\mathcal{D}_x(\Lambda)\Lambda,\]
are the same as in \cite{degondfrouvellemerino17}. The matrix $\mathcal{D}_x(\Lambda)\in\MM_3(\R)$ is defined as the unique matrix such that for all $\mathbf{w}\in\R^3$, and smooth functions $\Lambda:\R^3\to SO_3(\R)$, 
\[(\mathbf{w}\cdot\nabla_x)\Lambda=[\mathcal{D}_x(\Lambda)\mathbf{w}]_\times\Lambda\]
(see \cite[Section 4.5]{degondfrouvellemerino17}). Note that $C_3=C_2$ since the noise and alignement parameters which were denoted by $\nu$ and $d$ in \cite{degondfrouvellemerino17} have been taken equal to 1 here. Note also that these coefficients are functions of $\rho$ (through $\alpha$ only). The term $Y$ is an additional term which appears here due to the presence of the parameter $\alpha=\alpha(\rho)$ which is a function of $\rho$. It depends also on the derivative $\alpha'$ of $\alpha$~:
\[Y~:=\rho\alpha'\int_{SO_3(\R)} (\Lambda^TA-A^T\Lambda)M_{\alpha\Lambda}(A)\left(A\cdot \Lambda-\frac{3}{2}c_1(\alpha)\right)(\partial_t+Ae_1\cdot\nabla_x)\rho \,dA.\]
All the terms that involve the time derivative of $\rho$ are equal to zero since $\partial_t\rho$ does not depend on $A$ and with the change of variable $A' =\Lambda A^T\Lambda$ which has unit jacobian, it holds that $A'\cdot\Lambda=A\cdot\Lambda$ and therefore 
\begin{multline*}\int_{SO_3(\R)} (\Lambda^TA-A^T\Lambda)M_{\alpha\Lambda}(A)\left(A\cdot \Lambda-\frac{3}{2}c_1(\alpha)\right)\,dA\\=-\int_{SO_3(\R)} (\Lambda^TA'-A'^T\Lambda)M_{\alpha\Lambda}(A')\left(A'\cdot \Lambda-\frac{3}{2}c_1(\alpha)\right)\,dA'=0.\end{multline*}
We thus have: 
\[Y=Y_1+Y_2,\]
where
\[Y_1~: = \rho\alpha'\int_{SO_3(\R)} (A\cdot\Lambda)(Ae_1\cdot\nabla_x)\rho \,(\Lambda^TA-A^T\Lambda)M_{\alpha\Lambda}(A)\,dA,\]
and
\[Y_2~:= \frac{3}{2}c_1(\alpha)\rho\alpha'\int_{SO_3(\R)} (Ae_1\cdot\nabla_x)\rho\, (\Lambda^T A-A^T\Lambda)M_{\alpha\Lambda}(A)\,dA.\]
With the change of variable $A\mapsto \Lambda^T A$ these terms become
\[Y_1=\rho\alpha'\int_{SO_3(\R)} (\Lambda Be_1\cdot\nabla_x\rho)(B-B^T)(B\cdot I_3) M_{\alpha I_3}(B)\,dB,\]
\[Y_2~: = \frac{3}{2}c_1(\alpha)\rho\alpha' \int_{SO_3(\R)} (\Lambda Be_1\cdot \nabla_x\rho)(B-B^T)M_{\alpha I_3}(B)\,dB,\]
and they can be computed using the same techniques as in \cite{degondfrouvellemerino17} or lemma \ref{JAAg} in the appendix. More precisely, we can write 
\[\Lambda Be_1\cdot\nabla_x\rho = B\cdot R_{1,\rho},\]
where $R_{1,\rho}$ is the matrix, the first column of which is equal to $2\Lambda^T \nabla_x\rho$ and the others are all equal to zero. It satisfies:
\[\frac{R_{1,\rho}-R_{1,\rho}^T}{2}= [e_1\times\Lambda^T\nabla_x\rho]_\times.\]
By the change of variable $B\mapsto B^T$, we have 
\begin{multline*}
Y_1=\rho\alpha'\int_{SO_3(\R)} (B\cdot R_{1,\rho})(B-B^T)(B\cdot I_3) M_{\alpha I_3}(B)\,dB,\\
= \rho\alpha'\int_{SO_3(\R)} A\cdot(R_{1,\rho}-R_{1,\rho}^T)A(A\cdot I_3)M_{\alpha I_3}(A)\,dA.
\end{multline*}
This integral is of the form \eqref{PierreX} where 
\[g(A):=A\cdot I_3M_{\alpha I_3}(A)\]
is invariant by transposition and conjugation and 
\[J:=R_{1,\rho}-R_{1,\rho}^T\]
is a skew-symmetric matrix. From \eqref{PierreY} and \eqref{PierreZ} we get 
\[Y_1=\rho\alpha\mu(R_{1,\rho}-R_{1,\rho}^T)\]
with 
\[\mu:=\frac{1}{8}\int_{SO_3(\R)}(a_{21}-a_{12})^2\Tr(A)M_{\alpha I_3}(A)\,dA.\]
The first term $Y_1$ can therefore be written: 
\[Y_1=2\mu\rho\alpha'[e_1\times\Lambda^T\nabla_x\rho]_\times,\]
and using Rodrigues' formula we obtain: 
\[2\mu=\frac{1}{3}\Big\{ \left(1+2\cos\theta\right)\sin^2(\theta)\Big\}_{\alpha}\]
where $\{\cdot\}_\alpha$ has been defined in Proposition \ref{c1c2}.
Similarly the second term $Y_2$ can be written: 
\[Y_2=\frac{3}{2}c_1(\alpha) C_3\rho\alpha'[e_1\times\Lambda^T\nabla_x\rho]_\times,\]
where the coefficient $C_3$ is the same as in \cite{degondfrouvellemerino17}~: 
\[C_3=\frac{2}{3}\{\sin^2\theta\}_\alpha=C_2.\]
Finally, we obtain: 
\[Y=\rho\alpha'\left(\frac{3}{2}c_1(\alpha)C_3+\frac{1}{3}\Big\{ \left(1+2\cos\theta\right)\sin^2(\theta)\Big\}_{\alpha}\right)[e_1\times\Lambda^T\nabla_x\rho]_\times.\]

Putting all the terms together, we can conclude as in \cite{degondfrouvellemerino17}. First we notice that:
\[\Tr(L)=\delta_x(\Lambda),\,\,\,\,[\Lambda L^T e_1]_\times=[(\mathcal{D}_x(\Lambda)-[\mathbf{r}_x(\Lambda)]_\times)\Lambda e_1]_\times\,\,\,\,\text{and}\,\,\,\,[\Lambda Le_1]_\times\Lambda=\big((\Lambda e_1)\cdot\nabla_x\big)\Lambda.\]
Therefore we obtain by multiplying \eqref{XXX} by $\Lambda$ and dividing by $\alpha C_2$~:
\begin{multline*}
\rho(\partial_t\Lambda+c_2((\Lambda e_1)\cdot\nabla_x)\Lambda)+\tilde{c_3}[(\Lambda e_1)\times\nabla_x\rho]_\times\Lambda \\
+c_4\rho[-\mathbf{r}_x(\Lambda)\times(\Lambda e_1)+\delta_x(\Lambda)\Lambda e_1]_\times\Lambda=0,
\end{multline*}
where the coefficients
\[\tilde{c}_2:=\frac{C_4+C_5}{C_2}=\frac{1}{5}\frac{\{\sin^2\theta(2+3\cos\theta)\}_\alpha}{\{\sin^2\theta\}_\alpha}\,\,\,\,\text{and}\,\,\,\,c_4:=\frac{C_5}{C_2}=\frac{1}{5}\frac{\{\sin^2\theta(1-\cos\theta)\}_\alpha}{\{\sin^2\theta\}_\alpha}\]
are respectively equal to the coefficients $c_2$ and $c_4$ in \cite{degondfrouvellemerino17} and the coefficient $c_3$ in \cite{degondfrouvellemerino17} (which is equal to $1$) becomes: 
\[\tilde{c_3}=\frac{1}{\alpha}+\frac{\rho\alpha'}{\alpha}\left(\frac{3}{2}c_1(\alpha)+\frac{1}{2}\frac{\Big\{(1+2\cos\theta)\sin^2\theta\Big\}_\alpha}{\{ \sin^2\theta\}_\alpha}\right).\]

\section{Conclusion}

In this work, we have presented a new BGK model of body-attitude coordination where agents are described by a rotation matrix. Starting from the kinetic level (a space homogeneous BGK equation) we have drawn a parallel between our Vicsek-type model and the models of nematic alignment of polymers.  
We then have deduced the equilibria of the system and have shown a phase transition phenomenon triggered by the density of agents. Thanks to a gradient-flow structure specific to the BGK equation we have been able to describe the asymptotic behaviour of the system. Finally, we have derived the macroscopic models (SOHB) in the spatially inhomogeneous case. \\ 

On the modelling side, a rigorous mean-field limit which leads to the BGK equation is currently under study and will be the object of future work. However, many other questions remain open. At the kinetic level, our study relies on the dimension 3 and it would be interesting to extend the ideas developed here in $SO_n(\R)$, $n\geq3$, for example by drawing new parallels with higher dimensional polymers models or other similar models \cite{giacomin2012transitions,giacomin2012global}. In addition, the mathematical and numerical analyses of the macroscopic SOHB model are still in progress. \\ 

The BGK model studied here is a step towards the full description of the models of collective behaviour depicted in Figure \ref{models}. The tools and ideas that we have presented here may help to analyse other models of body-attitude coordination such as the non-normalised Fokker-Planck model in $SO_3(\R)$. Other models which take into account curvature control in addition to body-orientation, in the spirit of \cite{degond2011macroscopic}, could also be considered. 

\appendix
\section{Quaternions and rotations}\label{quaternionsandrotations}

These appendix is devoted to the proof of Proposition \ref{SO3quaternions}. We also give additional results about quaternions. The following lemma gives a link between quaternions and the theory of $Q$-tensors. 

\begin{lemma} Let $\SS_4^0(\R)$ be the space of symmetric $4\times4$ trace free matrices. If $Q\in \SS_4^0(\R)$ has two eigenvalues with eigenspaces of dimensions 1 and 3, then $Q$ can be written 
\[Q=\alpha\left(q\otimes q-\frac{1}{4}I_4\right),\]
for a given unit quaternion $q$ seen as a vector of $\R^4$. A matrix of this form is called a uniaxial $Q$-tensor. When $\alpha=1$ we will say that $Q$ is a normalised uniaxial $Q$-tensor. 
\end{lemma}

\begin{proof}
Let $Q\in \SS_4^0(\R)$ such that $Q$ has two eigenvalues with eigenspaces of dimensions 1 and 3. By the spectral theorem, there exists $P\in \mathcal{O}_3(\R)$ such that for a given $\alpha>0$~:
\[Q=\frac{\alpha}{4}P\diag(3,-1,-1,-1)P^T=\alpha P\diag(1,0,0,0)P^T-\frac{\alpha}{4} I_4,\]
and the result follows by taking $q$ equals to the first column of $P$. 
\end{proof}

\begin{proof}[\sc Proof of Proposition \ref{SO3quaternions}]
\begin{enumerate}[1.]
\item The group isomorphism $\Phi$ is explicitly computed in \cite{salamin1979application}. In particular, let $q\in \Sph^3/\pm1$. The matrix $A=\Phi(q)$ is defined for all purely imaginary quaternion $u\in\HH$ by $A[u]=[quq^*]$ where $(u_1,u_2,u_3)^T=:[u]\in\R^3$ is the vector associated to $u=u_1i+u_2j+u_3k$. More explicitly, if $q=x+iy+zj+tk$, then 
\[A=\left(\begin{array}{ccc}
x^2+y^2-z^2-t^2 & 2(yz-xt) & 2(xz+yt) \\ 
2(xt+yz) & x^2-y^2+z^2-t^2 & 2(zt-xy) \\ 
2(yt-xz) & 2(xy+zt) & x^2-y^2-z^2+t^2
\end{array}
\right).
\]
Note that we have identified $q\in\HH$ and its equivalence class in $\Sph^3/\pm1$ and that $\Phi$ is well defined since only quadratic expressions are involved. The fact that this group isomorphism is an isometry follows from \cite[Proposition A.3]{degondfrouvellemerinotrescases18}  and \cite[Lemma 4.2]{degondfrouvellemerino17}.
\item The expression
\[J\cdot A = \frac{1}{2} \Tr(\Phi(q)^T J)\]
is a quadratic form for $q$. We take $Q$ the matrix associated to this quadratic form. For $J=(J_{ij})_{i,j}$, using the explicit form of $A=\Phi(q)$ with $q=x+yi+zj+tk$ we obtain: 
\[Q=\frac{1}{4}\left(\begin{array}{cccc}
{J_{11}+J_{22}+J_{33}} & J_{32}-J_{23} & J_{13}-J_{31} & J_{21}-J_{12} \\ 
J_{32}-J_{23} & {J_{11}-J_{22}-J_{33}} & J_{12}+J_{21} & J_{13}+J_{31} \\
J_{13}-J_{31} & J_{12}+J_{21} & {-J_{11}+J_{22}-J_{33}} & J_{23}+J_{32} \\ 
J_{21}-J_{12} & J_{13}+J_{31} &  J_{23}+J_{32} & {-J_{11}-J_{22}+J_{33}}
\end{array}
\right).
\]
This is an isomorphism since $\dim \MM_3(\R)=\dim \SS_4^0(\R)$ and $J$ can be obtained from $Q$ similarly. Moreover, since 
the bilinear matrix associated to a quadratic form is uniquely defined, if $Q\in\SS_4^0(\R)$ is such that $\frac{1}{2}J\cdot\Phi(q)=q\cdot Qq$ for all $q\in\Sph^3/\pm1$, then $Q=\phi(J)$. 
\item To prove the third point, we note that a unit quaternion can be seen as a rotation in $\R^3$ in a more geometrical way (\cite[Section 5.1]{degondfrouvellemerinotrescases18}): for $\theta\in[0,\pi]$ and $\mathbf{n}=(n_1,n_2,n_3)^T\in\Sph^2$, let us define the unit quaternion $q$ by
\[q=\cos\frac{\theta}{2}+\sin\frac{\theta}{2}(n_1i+n_2j+n_3k),\]
The unit quaternion $q$ represents the rotation of angle $\theta\in[0,\pi]$ and axis $\mathbf{n}\in\Sph^2$ in the sense that if $R(\theta,\mathbf{n})\in SO_3(\R)$ denotes the matrix associated to the rotation of angle $\theta\in[0,\pi]$ around the axis $\mathbf{n}\in\Sph^2$ then $\Phi(q)=R(\theta,\mathbf{n})$. Note that $q$ and $-q$ represent the same rotation so $\Phi(q)$ is well defined by identifying $q$ with its equivalence class in $\Sph^3/\pm1$.\\

In $\R^3$ the composition of two rotations of respective angles and axis $(\theta,\mathbf{n})\in[0,\pi]\times\Sph^2$ and $(\theta',\mathbf{n}')\in[0,\pi]\times\Sph^2$ is itself a rotation: we have $R(\theta,\mathbf{n})R(\theta',\mathbf{n}')=R(\hat{\theta},\hat{\mathbf{n}})$ where the angle $\hat{\theta}\in[0,\pi]$ is defined by
\[\cos\frac{\hat{\theta}}{2}=\cos\frac{\theta}{2}\cos\frac{\theta'}{2}-\mathbf{n}\cdot\mathbf{n}'\sin\frac{\theta}{2}\sin\frac{\theta'}{2}.\]
Note that $\cos(\hat{\theta}/ {2})=q\cdot \bar{q}'$ where $q$ and $q'$ are the associated unit quaternions seen as vectors of dimension 4. In particular the dot product of two rotations matrices is
\[R(\theta,\mathbf{n})\cdot R(\theta',\mathbf{n}')=\frac{1}{2}\Tr\Big(R(\theta,\mathbf{n})R(\theta',-\mathbf{n}')\Big)=\frac{1}{2}(2\cos\tilde{\theta}+1)\]
where
\[\cos\frac{\tilde{\theta}}{2}=\cos\frac{\theta}{2}\cos\frac{\theta'}{2}+\mathbf{n}\cdot\mathbf{n}'\sin\frac{\theta}{2}\sin\frac{\theta'}{2}.\]
Besides, for the quaternions $q$ and $q'$ respectively associated to the rotations $R(\theta,\mathbf{n})$ and $R(\theta',\mathbf{n}')$, we have: 
\[q'\cdot Qq'=(q\cdot q')^2-\frac{1}{4}=\cos^2\frac{\tilde{\theta}}{2}-\frac{1}{4}=\frac{1}{4}(2\cos\tilde{\theta}+1),\]
where $Q$ is the normalised uniaxial $Q$-tensor: 
\[Q=q\otimes q-\frac{1}{4}I_4.\]
Finally $\frac{1}{2}R(\theta,\mathbf{n})\cdot R(\theta',\mathbf{n}')=q'\cdot Qq'$ and we obtain thanks to the previous point: 
\[\phi\Big(R(\theta,\mathbf{n})\Big)=Q,\]
that is to say: if $J\in SO_3(\R)$ then $\phi(J)$ is a normalised uniaxial $Q$-tensor. 
\item If $D=\diag(d_1,d_2,d_3)$ then using the explicit form of $\phi$ given in the second point: 
\[\phi(D)=\frac{1}{4}\left(\begin{array}{cccc}d_1+d_2+d_3 &  &  &  \\ & d_1-d_2-d_3 &  &  \\ &  & -d_1+d_2-d_3 &  \\ &  &  & -d_1-d_2+d_3\end{array}\right),\]
and if $Q=\diag(s_1,s_2,s_3,s_4)$ with $s_1+s_2+s_3+s_4=0$ then 
\[\phi^{-1}(Q)=2\left(\begin{array}{ccc}
{s_1+s_2} & & \\ 
 & {s_1+s_3} & \\
 & & {s_1+s_4}
 \end{array}
 \right).\]
\end{enumerate}
\end{proof}

\section{More about $SO_3(\R)$ and $SO_n(\R)$}\label{moreonson} 

\begin{lemma}\label{spanson} For all $n\geq3$~:
\[\Span(SO_n(\R))=\MM_n(\R).\]
\end{lemma}

\begin{proof}\
First we prove that the diagonal matrices form a subset of $\Span(SO_n(\R))$: it is enough to show that 
\[D~:=\left(\begin{array}{cccc}1 &  &  &  \\ & 0 &  &  \\ &  & \ddots &  \\ &  &  & 0\end{array}\right)\in \Span(SO_n(\R))\]
(the other diagonal matrices with only one nonzero coefficient can be obtained in a similar way). When $n$ is odd:
\[D=I_n + \left(\begin{array}{cccc}1 &  &  &  \\ & -1 &  &  \\ &  & \ddots &  \\ &  &  & -1\end{array}\right)\]
and both matrices in the sum are in $SO_n(\R)$. When $n\geq4$ is even, 

\[\left(\begin{array}{cc}0_4 &    \\  & I_{n-4}   \end{array}\right)=\frac{1}{2} I_n +\frac{1}{2}\left(\begin{array}{cc}-I_4 &    \\  & I_{n-4}    \end{array}\right)\in \Span(SO_n(\R)),\]
thus: 
\[\left(\begin{array}{ccccc}1 &   &   &   &   \\  & -1 &   &   &   \\  &   & -1 &   &   \\  &   &   & 1 &   \\  &   &   &   & 0_{n-4}\end{array}\right)=\left(\begin{array}{ccccc}1 &   &   &   &   \\  & -1 &   &   &   \\  &   & -1 &   &   \\  &   &   & 1 &   \\  &   &   &   & I_{n-4}\end{array}\right)-\left(\begin{array}{cc}0_4 &    \\  & I_{n-4}   \end{array}\right)\in \Span(SO_n(\R)),\]
and similarly: 
\begin{multline*}
4D=\left(\begin{array}{ccccc}1 &   &   &   &   \\  & 1 &   &   &   \\  &   & 1 &   &   \\  &   &   & 1 &   \\  &   &   &   & 0_{n-4}\end{array}\right)+\left(\begin{array}{ccccc}1 &   &   &   &   \\  & -1 &   &   &   \\  &   & -1 &   &   \\  &   &   & 1 &   \\  &   &   &   & 0_{n-4}\end{array}\right)
+\left(\begin{array}{ccccc}1 &   &   &   &   \\  & 1 &   &   &   \\  &   & -1 &   &   \\  &   &   & -1 &   \\  &   &   &   & 0_{n-4}\end{array}\right)
\\+\left(\begin{array}{ccccc}1 &   &   &   &   \\  & -1 &   &   &   \\  &   & 1 &   &   \\  &   &   & -1 &   \\  &   &   &   & 0_{n-4}\end{array}\right)\in \Span(SO_n(\R)).\end{multline*}
The SSVD (Definition \ref{ssvd}) gives the result for any matrix. 
\end{proof}

\begin{corollary}\label{commute}
For $n\geq3$, a matrix that commutes with any matrix of $SO_n(\R)$ is of the form $\lambda I_n$.
\end{corollary}

We can now prove Lemma \ref{JAA} and its generalisation Lemma \ref{JAAg}.

\begin{proof}[\sc Proof (of Lemma \ref{JAA})] The linear map $\Phi~: \MM_n(\R)\to \MM_n(\R)$ defined by 
\[\Phi(J)~:=\int_{SO_n(\R)} (J\cdot A)A\,dA,\]
satisfies for all $P\in SO_n(\R)$, 
\[\Phi(P)=P\Phi(I_n)=\Phi(I_n)P,\] 
which by corollary \ref{commute} means that 
\[\Phi(I_n)=\lambda I_n.\]
Therefore, $\Phi(P)=\lambda P$ for any $P\in SO_n(\R)$ and the same is true for any matrix $J\in\MM_n(\R)$ by lemma \ref{spanson}. To compute $\lambda$, notice that for the matrix $e_i\otimes e_j$~: 
\[\frac{1}{2}\int_{SO_n(\R)} (e_i\cdot Ae_j)^2\,dA=\lambda.\]
Summing this equality for all $i,j$ gives: 
\[\frac{n}{2}=\lambda n^2,\]
and $\lambda=1/2n$.
\end{proof}

\begin{proof}[\sc Proof (of Lemma \ref{JAAg})]
Let us define the linear map $\psi~: \MM_n(\R)\longrightarrow\MM_n(\R)$ by
\begin{equation}\label{PierreX}\psi(J):=\int_{SO_n(\R)} (J\cdot A)A\,g(A)\,dA.\end{equation}
\begin{enumerate}
\item We first note that $\psi$ is self-adjoint for the dot product $A\cdot B=\Tr(A^TB)$: for any $K\in\MM_n(R)$, 
\[\psi(J)\cdot K =\int_{SO_n(\R)} (J\cdot A)(K\cdot A)\,g(A)\,dA=J\cdot\psi(K).\]
\item We prove that $\Vect(I_n)$ is a stable supspace for $\psi$~: for any $P\in SO_n(\R)$ we have: 
\[P\psi(I_n) P^T=\int _{SO_n(\R)} (I_n\cdot A) PAP^T\,g(A)\,dA=\psi(I_n),\]
and we conclude with corollary \ref{commute} that $\psi(I_n)=\alpha I_n$ with: 
\[\alpha=\frac{2}{n} I_n\cdot \psi(I_n)=\frac{1}{2n}\int_{SO_n(\R)} \Tr(A)^2 g(A)\,dA.\]
\item Since $\psi$ is a self adjoint operator, the orthogonal subspace $\Vect(I_n)^\perp$ is also a stable subspace. Moreover, using the change of variable $A\mapsto A^T$, we see that $\psi(J^T)=\psi(J)^T$ and we have the decomposition: 
\[\Vect(I_n)^\perp=\SS_n^0(\R)\overset{\perp}{\oplus}\mathscr{A}_n(\R),\]
where $\SS_n^0(\R)$ and $\mathscr{A}_n(\R)$ are respectively the subspace of trace free symmetric matrices and the subspace of skew-symmetric matrices. They are both stable subspaces. 
\item We prove now that $\psi~: \SS_n^0(\R)\to\SS_n^0(\R)$ is a uniform scaling. By the spectral theorem, every matrix $J\in \SS_n^0(\R)$ can be written 
\[J=PDP^T,\]
where $P\in SO_n(\R)$ and $D$ is diagonal. Since 
\[\psi(PDP^T)=P\psi(D)P^T,\]
it is enough to prove that there exists $\lambda\in \R$ such that for all diagonal matrices $D$, $\psi(D)=\lambda D$. For $D=\diag(d_1,\ldots,d_n)$, we have: 
\[\psi(D)=\frac{1}{2}\sum_{k=1}^{n} d_k \int_{SO_{n}(\R)} a_{kk} A\,g(A)\,dA.\]
Now, if $i\ne j$, let $D^{ik}\in SO_n(\R)$ be the diagonal matrix such that all the coefficients are equal to 1 except $D^{ik}_{ii}$ and $D^{ik}_{kk}$ which are equal to $-1$. Then the change of variable $A\mapsto D^{ik}A(D^{ik})^T$ gives  
\[\int_{SO_{n}(\R)} a_{kk} a_{ij}\,g(A)\,dA = -\int_{SO_{n}(\R)} a_{kk}\,a_{ij}\,g(A)\,dA  =0,\]
which proves that $\psi(D)$ is diagonal and the $i$-th coefficient of $\psi(D)$ is: 
\[\psi(D)_{ii}=\frac{1}{2}\sum_{k=1}^{n} d_k \int_{SO_{n}(\R)} a_{kk}\,a_{ii}\,g(A)\,dA.\]
Using the fact that $\Tr(D)=0$ and that all the $\int_{SO_n(\R)} a_{ii}\,a_{kk}\,g(A)\,dA$ are equal for $i\ne k$ (by using conjugation by the matrices $P^{ij}$, $i,j\ne k$, see Definition \ref{changeofvariable}), we obtain: 
\[\psi(D)_{ii}=\frac{1}{2}d_i \int_{SO_{n}(\R)} (a_{11}^2-a_{11}\,a_{22}) \,g(A)\,dA,\]
and we conclude that for all $J\in \SS_n^0(\R)$~: 
\[\psi(J)=\lambda J,\]
with 
\[\lambda=\frac{1}{2}\int_{SO_{n}(\R)} (a_{11}^2-a_{11}\,a_{22}) \,g(A)\,dA=\frac{1}{4}\int_{SO_n(\R)} (a_{11}-a_{22})^2\,g(A)\,dA.\]
\item We prove similarly that $\psi~: \mathscr{A}_n(\R)\to\mathscr{A}_n(\R)$ is a uniform scaling. Every $J\in\mathscr{A}_n(\R)$ can be written 
\[J=P CP^T,\]
where 
\begin{equation}\label{PierreA}C=\left(\begin{array}{ccccccc} C_1 & & & & & & \\ & C_3 & & & & & \\ & & \ddots &  & & & \\ & & & C_{2p-1}  & & & \\ & & & & 0 & & \\ & & & & & \ddots & \\ & & & & & & 0\end{array}\right)\end{equation}
is a block diagonal matrix with blocks 
\[C_i=\left(\begin{array}{cc} 0 & -c_i \\ c_i & 0\end{array}\right),\,\,\,\,c_i\in\R^*,\,\,\,\,\,i\in\{1,3,5,\ldots,2p-1\}\]
so that, 
\[\psi(J)=\psi(PCP^T)=P\psi(C)P^T,\]
and 
\[\psi(C)=\frac{1}{2}\sum_{k=1}^{p} c_{2k-1} \int_{SO_n(\R)} (a_{2k,2k-1}-a_{2k-1,2k})\,A\,g(A)\,dA.\]
When $n\geq3$, by using conjugation by the matrices $D^{2j-1,2j}$ for $j\in\{1,\ldots,\floor{n/2}\}$  (Definition \ref{changeofvariable}) we can see that for each $k\in\{1,\ldots,p\}$, the matrix
\[M_k:=\int_{SO_n(\R)} (a_{2k,2k-1}-a_{2k-1,2k})\,A\,g(A)\,dA\]
is of the form \eqref{PierreA}. Moreover, when $n\geq5$, by using conjugation by matrices  $D^{2j-1,2\ell-1}$ for $j\ne\ell$, $j,\ell\in\{1,\ldots,\floor{n/2}\}$ and $j,\ell\ne k$, we can see that all the diagonal blocks of $M_k$ are equal to zero except the one in position $2k-1$. When $n=3$ there is only one block in position 1 so the result holds but when $n=4$ such $j$ and $\ell$ do not exist. In conclusion, when $n\ne4$, $\psi(C)$ is of the form \eqref{PierreA} and each diagonal block $C_{2k-1}'$ of $\psi(C)$ is written 
\[C_{2k-1}'=\mu_{2k-1}C_{2k-1}\]
with 
\[\mu_{2k-1}:=\int_{SO_n(\R)} (a_{2k,2k-1}-a_{2k-1,2k})a_{2k,2k-1}g(A)\,dA=\int_{SO_n(\R)} (a_{21}-a_{12})a_{21}g(A)\,dA,\]
where this equality follows by using conjugation by ``block-permutation matrices": 
\[Q^k := \left(\begin{array}{ccccccccccc}I_2 &  &  &  &  &  &  &  &  &  &  \\ & \ddots &  &  &  &  &  &  &  &  &  \\ &  & I_2 &  &  &  &  &  &  &  &  \\ &  &  & 0_2 & -I_2 &  &  &  &  &  &  \\ &  &  & I_2 & 0_2 &  &  &  &  &  &  \\ &  &  &  &  & I_2 &  &  &  &  &  \\ &  &  &  &  &  & \ddots &  &  &  &  \\ &  &  &  &  &  &  & I_2 &  &  &  \\ &  &  &  &  &  &  &  & 1 &  &  \\ &  &  &  &  &  &  &  &  & \ddots &  \\ &  &  &  &  &  &  &  &  &  & 1\end{array}\right),\]
where the first zero on the diagonal is in position $2k-1$. Therefore, 
\begin{equation}\label{PierreY}\psi(C)=\mu C,\end{equation}
with
\begin{equation}\label{PierreZ}\mu=\frac{1}{2}\int_{SO_n(\R)}(a_{21}-a_{12})a_{21}\,g(A)\,dA=\frac{1}{4}\int_{SO_n(\R)} (a_{21}-a_{12})^2\,g(A)\,dA.\end{equation}
\item Finally, for $J\in\Vect(I_n)^\perp$, writing 
\[J=\frac{J+J^T}{2}+\frac{J-J^T}{2},\]
we have: 
\[\psi(J)= \beta J+\gamma J^T,\]
with 
\[\beta=\frac{1}{2}(\lambda+\mu)=\frac{1}{4}\int_{SO_n(\R)} \Big((a_{11}^2-a_{11}\,a_{22})+(a_{21}-a_{12})a_{21}\Big)\,g(A)\,dA,\]
and
\[\gamma=\frac{1}{2}(\lambda-\mu)=\frac{1}{4}\int_{SO_n(\R)} \Big((a_{11}^2-a_{11}\,a_{22})-(a_{21}-a_{12})a_{21}\Big)\,g(A)\,dA.\]
\item In conclusion, writing the decomposition 
\[J=\frac{1}{n}\Tr(J)I_n+K,\]
where $K\in\Vect(I_n)^\perp$, we obtain
\[\psi(J)=a\Tr(J) I_n+bJ+cJ^T,\]
with
\[a=\frac{\alpha-\beta-\gamma}{n},\]
and
\[b=\beta=\frac{1}{8}\int_{SO_n(\R)}\Big((a_{11}-a_{22})^2+(a_{12}-a_{21})^2\Big)\,g(A)\,dA,\]
and
\[c=\gamma=\frac{1}{8}\int_{SO_n(\R)}\Big((a_{11}-a_{22})^2-(a_{12}-a_{21})^2\Big)\,g(A)\,dA.\]
And there are of course many other ways to write the coefficients $a$, $b$ and $c$.
\end{enumerate}
\end{proof}

\begin{rem}
In dimension 4, the result still holds for symmetric matrices. For general matrices, the result can be proved in particular cases, for instance when $g$ is a function of the trace, by using an explicit parametrisation of $SO_4(\R)$ such as the 4-dimensional version of \eqref{SOSphere}. 
\end{rem}

\section*{Acknowledgments}
\addcontentsline{toc}{section}{Acknowledgments}

PD acknowledges support by the Engineering and Physical Sciences Research Council (EPSRC) under grants no. EP/M006883/1 and EP/P013651/1, by the Royal  Society and the Wolfson Foundation through a Royal Society Wolfson Research Merit Award no. WM130048 and by the National Science Foundation (NSF) under grant no. RNMS11-07444 (KI-Net). PD is on leave from CNRS, Institut de Math\'ematiques de Toulouse, France.

\noindent A. D. acknowledges the hospitality of CEREMADE, Universit\'e Paris-Dauphine and the University of Sussex where part of this research was carried out. A.D. was supported for this work by the \'Ecole Normale Sup\'erieure de Rennes through an Erasmus+ grant. This work was initiated at the \'Ecole Normale Sup\'erieure de Lyon as part of a Master degree delivered to A.D. 

\noindent A.F. acknowledges support from the EFI project ANR-17-CE40-0030 and the Kibord project ANR-13-BS01-0004 of the French National Research Agency (ANR), from the project Défi S2C3 POSBIO of the interdisciplinary mission of CNRS, and the project SMS co-funded by CNRS and the Royal Society.

\noindent SMA is supported by the Vienna Science and Technology Fund (WWTF) with a Vienna Research Groups for Young Investigators, grant VRG17-014.\\

\noindent\textit{Data statement:} no new data were collected in the course of this research. \\

\noindent\textit{Conflict of interest:} the authors have no conflict of interest to declare.

\bibliographystyle{abbrv}
\bibliography{biblio}

\end{document}